\documentclass[12pt]{amsart}
\usepackage{amsmath,amssymb,amscd,colordvi}
\usepackage{color}

\newtheorem{theorem}{Theorem}[section]
\newtheorem{corollary}{Corollary}[section]
\newtheorem{lemma}{Lemma}[section]

\newtheorem{remark}{Remark}[section]

\newtheorem{pro}{Proposition}[section]

\newcommand{\RR}{{\mathbb R}}

\newcommand{\ep}{\hfill {$\Box$}}

\begin{document}

\title[Resonances rank one]
{Resonances under rank one perturbations}

\author[Bourget]{Olivier Bourget*}

\author[Cort\'es]{V\'ictor H. Cort\'es*}
\author[Del R\'io]{Rafael del R\'io**}
\author[Fern\'andez]{Claudio Fern\'andez*}

\address{*Pontificia Universidad Cat\'olica de Chile. Facultad de Matem\'aticas}
\address{**IIMAS-UNAM, Department of Mathematical Physics, 04510 CDMX, Mexico}

\email{delrio@iimas.unam.mx}
\email{cfernand@mat.uc.cl}
\email{bourget@mat.uc.cl}
\email{vcortes@mat.uc.cl}

\date{22feb2017}

\thanks{Fondecyt Grants 1141120 and 1161732, ECOS/CONICYT Grant
C10E01.}
\thanks{Research partially supported by UNAM-DGAPA-PAPIIT IN105414}


\begin{abstract}
We study resonances generated by rank one perturbations of selfadjoint operators with eigenvalues embedded in the continuous spectrum. Instability
of these eigenvalues is analyzed and almost exponential decay for the associated resonant states is exhibited. We show how these results can be
applied to  Sturm-Liouville operators. Main tools are the Aronszajn-Donoghue theory for rank one perturbations, a reduction process of the resolvent
based on Feshbach-Livsic formula, the Fermi golden rule and a careful analysis of the Fourier transform of quasi-Lorentzian functions. We relate
these results to sojourn time estimates and spectral concentration phenomena
\end{abstract}
%\date{\today}
\maketitle

%%%%%%%%%%%%%%%%%%%%

\section{Introduction}

The resonance phenomenon appears in several areas of physics and mathematics such as Classical, Quantum and Wave Mechanics. Several attempts have been done to give it a precise mathematical description. We refer to \cite{simon2} for a discussion about the difficulties arising in characterizing rigorously the concept of resonance for autonomous systems in  Quantum Mechanics.

One of the most fruitful approaches consists in defining quantum resonances as poles of a suitable meromorphic continuation of the Hamiltonian resolvent, from the upper half complex plane to the lower half plane. Each pole appears as an "eigenvalue" with negative imaginary part, corresponding to generalized eigenfunctions outside the Hilbert space. There is a huge literature on this subject and we refer the reader to \cite{hislop} and references therein.

Resonances can also be characterized in terms of time exponential decay of the time evolution of the system governed by the Hamiltonian (defined as a self-adjoint operator on some Hilbert space $\mathcal H$). This behaviour can be traced by means of the survival probability $P_{\varphi}$ for some suitable states $\varphi$. This quantity defined by
$$
P_{\varphi}(t)=|\langle \varphi, e^{-iHt}\varphi \rangle |^2 \, ,
$$
measures the probability of finding at time $t$ the system governed by the Hamiltonian $H$ in its initial state $\varphi$. On one hand, we know that exact exponential decay is impossible for many models of physical interest, see e.g. \cite{simon2}. On the other hand, if $z=\lambda -i\Gamma$ ($\Gamma >0$) is a pole of the resolvent of the Hamiltonian $H$ with "resonant eigenfunction" $\varphi$ (i.e. $H\varphi=z \varphi$), we would formally expect that,
$$
P_{\varphi}(t)=e^{-2 \Gamma t} \|\varphi\|^2 \, ,
$$
which is incorrect if the resonant eigenfunction $\varphi$ does not belong to the Hilbert space. Thus, in presence of a resonance $z$, the best one can hope is the existence of a state $\psi \in \mathcal H$ such that the quantity $\langle \psi ,e^{-iHt}\psi \rangle$ behaves approximately as $e^{-izt}$. Both these quantities equal $1$ at $t=0$ and in most cases of interest, both approach to zero as $t$ tends to $\infty$. The main objective is then to estimate the difference,
\begin{equation*}
%\label{time}
\langle \psi, e^{-iHt}\psi \rangle \,-\, e^{-izt}\, ,
\end{equation*}
for $t$ not to close to $0$ nor to $\infty$.

For differential operators on the real half line, this difference can be estimated uniformly in time \cite{lavine} or in $L^2$ norm \cite{astaburuaga} by means of ODE techniques. In these cases, the function $\psi$ is a truncated resonant eigenfunction. Pointwise estimates have been exhibited when the resonance appears with the perturbation of an instable simple eigenvalue embedded in some continuous spectrum, see e.g. \cite{cattaneo1} and \cite{jensen} for a review. The main ingredients are in that case the Feshbach-Livsic reduction and the Fermi Golden rule. In \cite{cattaneo1}, this approach is actually combined with some positive commutator techniques (Mourre theory) and the estimates are obtained once the eigenfunction is localized in energy.

The consistent use of the Feshbach-Livsic reduction to study resonances can be traced back at least to \cite{How} and has been a source of several results in the last decades in different areas. In particular, it has been used consistently in spectral theory for Non-Relativistic Quantum Electrodynamics since \cite{bfs1} and \cite{bfs2}. For the relationship between time evolution (the perspective we address in this paper) and poles of the resolvent in the context of analytic perturbation theory, see also \cite{hhh}, \cite{fgs} and references therein.

In this paper, we adapt the Feshbach-Livsic reduction to the context of differential operators on the real half line and pointwise estimates are exhibited when the resonance appears with the perturbation of an instable simple eigenvalue embedded in the absolutely continuous spectrum of such operators. Although various tools developed here can be easily adapted to a fairly wide class of perturbations, we have decided to narrow our discussion to the rank-one case and to relate these results with classical results in this field \cite{don}, \cite{simon1}. We intend to  propose several extensions in a forthcoming paper.

We start in Section 2 by establishing conditions which ensure that the Fourier transform of a Lorentz-like function exhibits approximate exponential time decay. The proof of Theorem \ref{protothm} is based on techniques of classical analysis, we have mainly singled out from \cite{cattaneo1}. The development of this section is independent from the rest of the paper. In Section 3, we turn our attention on rank one perturbations of the form
$$
H_{\kappa}=H_0 + \kappa |\psi \rangle  \langle \psi | \, ,
$$
where $H_0$ has a simple eigenvalue embedded in some absolutely continuous spectrum. In Theorems \ref{absence} and \ref{pp}, we show how the instability of the embedded eigenvalue and the spectral properties of the operators $H_{\kappa}$ are related to the boundary values of the reduced resolvent of $H_0$ and the Fermi Golden Rule.
%A related result can be found in \cite{simon3}.
Next, we prove Theorem \ref{thmrankone}, which formalizes the existence of a resonance in term of almost exponential decay in the case of rank one perturbations under suitable hypotheses on the reduced resolvent of $H_0$. The proof combines the Feshbach-Livsic reduction process and Krein's formula with Theorem \ref{protothm}. Corollary \ref{spectralconc} discusses the relationships with Kato's spectral concentration. In Corollary \ref{sojourn}, we deduce the asymptotics for the sojourn time of the corresponding  eigenstate under the evolution governed by $H_{\kappa}$ for small values of $\kappa$. Finally, we show in Section 4, how the boundary properties of the reduced resolvent of $H_0$ can be deduced from the properties of the spectral measure of $H_0$ when it has finite multiplicity. This reformulation of the problem is summed up in Theorem \ref{abstract} and its proof makes essential use of the properties of the Borel transform, see  Section \ref{Boundary limit of Borel transforms}. All these results are illustrated in Section 5 by a Sturm-Liouville model. In contrast with \cite{cattaneo1}, the point of view developed in this paper does not require any positive commutator techniques.

We shall use the notation $C^{1,\beta}(I)$ for the set of functions with first derivative $\beta$- H\"older continuous in $I$. Given a function of complex variable $F(z)$, we write $F(\lambda +i0)$ for $\lim_{\epsilon \downarrow 0}F(\lambda +i\epsilon)$. For the spectral family of orthogonal proyections of an operator $T$ we shall write $E_T$ 
and $\rho(T)$,  $\sigma (T)$,  $\sigma_{p}(T)$,  $\sigma_{sc}(T), \sigma_{ac}(T)$ denote the resolvent, the spectrum, the eigenvalues, the singular continuous and absolutely continuous spectra of $T$ respectively \cite{RSV1}. The characteristic function of a Borel set $\Delta $  will be denoted as usual by $\chi_{\Delta} (x)$ where $\chi_{\Delta} (x) =1 $ if $x\in \Delta$ and $0 $ if $x\notin \Delta$. $\mathbb{R}$ stands for the real numbers, $\mathbb R^+$  the non negative reals and $\mathbb{C}$ for  the complex numbers. $\Im z, \Re z$ stand for the imaginary and real parts of $z$.

%%%%%%%%%%%%%%%%%%

\section{Almost exponential decay}

Theorem \ref{protothm} is the main result of this section and it is the first ingredient in the analysis developed in this paper. It provides some estimates on the Fourier transform of families of Lorentz-like functions defined on ${\mathbb R}$ by:
$$
\lambda \mapsto g(\lambda)\, \Im \left(\frac{1}{\lambda_{\kappa} -\lambda -\kappa^2 F(\lambda, \kappa)}\right) \, ,
$$
with $\kappa \in [-\kappa_0,\kappa_0]$ for some $\kappa_0 >0$ and under suitable assumptions set on the functions $g$, $F$ and the family of real numbers $(\lambda_{\kappa})_{\kappa \in [-\kappa_0,\kappa_0]}$. In the following, $g\in C_0^{\infty}({\mathbb R})$ is real-valued, compactly supported on $(a,b)$ for some $-\infty < a< b< \infty$, and we also assume that $0 \leq g\leq 1$ and $g\equiv 1$ on $[a_0 , b_0]$ for some $a< a_0 <b_0 <b$. In addition,
\begin{itemize}
\item[({\bf H0})]
\begin{itemize}
\item[(a)] $\lim_{\kappa \rightarrow 0}\lambda_{\kappa} =\lambda_0$, $\lambda_0 \in (a_0,b_0)$,
\item[(b)] the complex-valued function $F$ is bounded on $[a,b]\times [-\kappa_0,\kappa_0]$ and continuous at the point $(\lambda_0, 0)$,
\end{itemize}
\item[({\bf H1})] for any $\kappa \in [-\kappa_0,\kappa_0]$, the function $F(\cdot, \kappa)$ is $C^1$ on $[a,b]$ and
\begin{itemize}
\item[(a)] the function $F':= \partial_{\lambda} F$ is bounded on $[a,b]\times [-\kappa_0,\kappa_0]$,
\item[(b)] for any $\kappa \in [-\kappa _0,\kappa_0]$, the function $F(\cdot, \kappa)$ is $C^{1,\alpha}$ on $[a_0,b_0]$, uniformly in $\kappa \in[-\kappa_0,\kappa_0]$ for some $\alpha \in (0,1]$,
\end{itemize}
\item[({\bf H2})] for any $\kappa \in [-\kappa_0,\kappa_0]$, $\inf_{\lambda \in [a_0,b_0]} \Im F(\lambda, \kappa) >0$.
\end{itemize}

\noindent{\bf Remark}: In Assumption ({\bf H1}), $F_{\kappa}'$ is defined at $a$ and $b$ by taking the corresponding lateral derivatives. If the function $F$ is continuous in both variables on $[a,b]\times [-\kappa_0,\kappa_0]$, it is necessarily bounded. Note finally that if $\Im F(\lambda_0,0) >0$ and $F$ is continuous at the point $(\lambda_0,0)$, we can deduce the existence of an interval $[a_0,b_0]$, on which ({\bf H2}) holds for small values of $\kappa$. The condition $\Im F(\lambda_0,0) >0$ is known as the Fermi Golden Rule.

We have that:
\begin{theorem}\label{protothm} Assume (\,{\bf H0}), (\,{\bf H1}) and (\,{\bf H2}) hold. Then, given any $0 <\delta <\min (|\lambda_0 -a_0 |, |b_0- \lambda_0 |,1)$, and $\kappa \neq 0$ small enough,
\begin{itemize}
\item[(a)] there exists a unique solution in $[a_0,b_0]$ to the equation: $\lambda=\lambda_{\kappa}-\kappa^2 \Re F_{\kappa}(\lambda)$, denoted by $\lambda_{\kappa}^{\infty}$, which satisfies: $|\lambda_{\kappa}^{\infty} - \lambda_{\kappa}| \leq C \kappa^2$ for some $C>0$ and $a_0 +\delta \leq \lambda_{\kappa}^{\infty} \leq b_0 -\delta $,
\item[(b)] for all $t\in {\mathbb R}$,
\begin{equation}\label{quasiexpo}
\frac{1}{\pi}\, \int_{-\infty}^{\infty} d\lambda\,  e^{-i\lambda t}g(\lambda)\, \Im \left(\frac{1}{\lambda_{\kappa} -\lambda -\kappa^2 F(\lambda, \kappa)}\right) = c_{\kappa} \, e^{-i\zeta_{\kappa} |t|} +R(t,\kappa)
\end{equation}
where $c_{\kappa}^{-1} = 1+\kappa^2 F'(\lambda_{\kappa}^{\infty},\kappa),$
\begin{equation}\label{zetax}
\zeta_{\kappa} = \lambda_{\kappa}^{\infty} -i \kappa^2 c_{\kappa} \Im F(\lambda_{\kappa}^{\infty},\kappa)\, ,
\end{equation}
\end{itemize}
and the error term $R(t,\kappa)$ satisfies: $|R(t,\kappa)|\le C\kappa^{2}$ if $\alpha \in (0,1]$, $|t| |R(t,\kappa)|\le C\kappa^{2\alpha}$ if $\alpha \in (0,1)$ and
 $|t| |R(t,\kappa)|\le C\kappa^{2}|\ln | \kappa ||$ if $\alpha =1$.

\end{theorem}

\begin{remark} By combining (\ref{asympt}) and Hypothesis ({\bf H0}) in Theorem \ref{protothm}, we also deduce that: $\lim_{\kappa \rightarrow 0} \Re \zeta_{\kappa} = \lambda_0$ and
$$\lim_{\kappa \rightarrow 0} \frac{1}{\kappa^2} \Im \zeta_{\kappa} = - \Im F(\lambda_0,0) < 0 \, .$$
 In particular we obtain using \eqref{zetax} that
 \begin{gather}\label{asympt}
\lim_{\kappa \rightarrow 0} \frac{\lambda_{\kappa}^{\infty} - \Re \zeta_{\kappa}}{\kappa^2 \Im F(\lambda_\kappa ^{\infty},\kappa)} = 1 \quad \text{and}\quad \lim_{\kappa \rightarrow 0} \frac{\Im \zeta_{\kappa} }{\kappa^2 \Im F(\lambda_\kappa ^{\infty},\kappa)} = -1\, .
\end{gather}

\end{remark}

\begin{remark} The model for the integral described in Theorem \ref{protothm} is the Fourier transform of Lorentzian functions. Let $\mu \in {\mathbb R}$ and $\Gamma > 0$. Then for any $\lambda \in {\mathbb R}$,
$$\Im (\frac{1}{\mu -\lambda -i\Gamma}) = \frac{\Gamma}{(\lambda-\mu)^2 + \Gamma^2}$$
and for any $t\in {\mathbb R}$,
$$\frac{1}{\pi}\, \int_{-\infty}^{\infty} d\lambda\,  e^{-i\lambda t}\,\Im \left(\frac{1}{\mu -\lambda - i\Gamma}\right) = e^{-i(\mu-i\Gamma) |t|} \, ,$$
which decays exponentially at infinity. This observation is one of the key to the proof of Theorem \ref{protothm}.
\end{remark}

The strategy for the proof of Theorem \ref{protothm} follows essentially \cite{cattaneo1}. The fixed point argument has been borrowed to \cite{king}.

%*************************************************
\subsection{Proof of Theorem \ref{protothm}}

For simplicity, let us write for any $(\lambda,\kappa) \in [a,b]\times [-\kappa_0,\kappa_0]$, $D(\lambda,\kappa) = \lambda_{\kappa} -\lambda -\kappa^2F(\lambda,\kappa)$. Due to Hypothesis ({\bf H2}), we have that for any $\lambda \in [a_0,b_0]$, $0<|\kappa |\leq \kappa_0$, $|D(\lambda,\kappa) | \geq \kappa^2 \inf_{\lambda \in [a_0,b_0]} \Im F_{\kappa}(\lambda) > 0$. Now, fix $\delta_1 \in (0, \min (|\lambda_0 -a_0 |, |b_0- \lambda_0 |,1))$. According to Hypotheses ({\bf H0}) and ({\bf H1})(a), we can pick $0 <\kappa_1 \leq \kappa_0$ such that:
\begin{itemize}
\item Ran $(\lambda_{\kappa} - \kappa^2 \Re F(\cdot , \kappa)) \subset [a_0+\delta_1, b_0-\delta_1]$, for any $|\kappa |\leq \kappa_1$,
\item $\kappa_1^2 \sup_{(\lambda,\kappa) \in [a_0,b_0]\times [-\kappa_0,\kappa_0]} |F'(\lambda, \kappa)| \leq 1-\delta_1$.
\end{itemize}
In particular, for any $\lambda \in [a,a_0]\cup [b_0,b]$, $|\kappa |\leq \kappa_1$, $|D(\lambda,\kappa) |\geq |\lambda_{\kappa} -\lambda -\kappa^2 \Re F_{\kappa}(\lambda)| \geq \delta_1$. This allows us to define the function $G$ for $\lambda \in [a,b]$ and $0 < |\kappa |\leq \kappa_1$ by: $G(\lambda,\kappa)=D(\lambda,\kappa)^{-1}$ and then to give a sense to the integral
\begin{equation}\label{quantity}
{\mathcal I}(t,\kappa) = \frac{1}{\pi}\, \int_{-\infty}^{\infty} d\lambda\,  e^{-i\lambda t}g(\lambda)\, \Im \left(\frac{1}{\lambda_{\kappa} -\lambda -\kappa^2 F(\lambda, \kappa)}\right) \, ,
\end{equation}
for $t\in {\mathbb R}$ and $0< |\kappa |\leq \kappa_1$. Since ${\mathcal I}(-t,\kappa)=\overline {{\mathcal I}(t,\kappa)}$, it is enough to prove Theorem \ref{protothm} for $t\geq 0$ and $0< |\kappa |\leq \kappa_1$. The proof consists in adding and subtracting the Lorentz-like function $\Im G_1$ (which is explicited later on) on the interval $[a_0, b_0]$ and reduces the problem to study
$$
{\mathcal I}_1(t,\kappa):= \frac{1}{\pi}\int_{a_0}^{b_0} d\lambda \, e^{-i\lambda t} \Im \,G_1(\lambda, \kappa) \, .
$$
This integral contributes to the quasi-exponential behaviour term in (\ref{quasiexpo}) while the remainder terms contribute to the error term $R$.

Given $\delta_1$ as before and $|\kappa |\leq \kappa_1$, we define first by a fixed point argument the real number $\lambda_{\kappa}^{\infty}$ for $\kappa \in [-\kappa_1,\kappa_1]$:
\begin{lemma}\label{fixed} Given any $\kappa\in [-\kappa_1,\kappa_1]$, there is a unique solution to the equation: $\lambda=\lambda_{\kappa}-\kappa^2 \Re F(\lambda,\kappa)$ in $[a_0,b_0]$. Actually, if $\lambda_{\kappa}^{\infty}$ denotes this solution, we have that: $\lambda_{\kappa}^{\infty}\in [a_0+\delta_1,b_0 -\delta_1]$ and $|\lambda_{\kappa}^{\infty} -\lambda_{\kappa} | \leq \kappa^2 \sup_{(\lambda,\kappa) \in [a_0,b_0]\times [-\kappa_0,\kappa_0]} |F(\lambda,\kappa)|$.
\end{lemma}
\begin{proof} By hypothesis, for any $|\kappa |\leq \kappa_1$, $\kappa^2 \sup_{(\lambda,\kappa) \in [a_0,b_0]\times [-\kappa_0,\kappa_0]} |F'(\lambda,\kappa)| <1$ and Ran $(\lambda_{\kappa} - \kappa^2 \Re F(\cdot , \kappa)) \subset [a_0+\delta_1, b_0-\delta_1]$. Therefore, given such $\kappa\in [-\kappa_1,\kappa_1]$, the function $\lambda \mapsto \lambda_{\kappa} - \kappa^2\, \Re F(\lambda,\kappa)$ maps $[a_0,b_0]$ (resp. $[a_0+\delta_1,b_0-\delta_1]$) on itself and is strictly contractive. So, we apply the Banach fixed point theorem, and the conclusions follow.
\end{proof}

This proves statement (a) of Theorem \ref{protothm}. Now, let us define for any $\lambda \in [a_0,b_0]$, $|\kappa |\leq \kappa_1,$
\begin{align*}
\widehat{D}(\lambda,\kappa) &=\lambda_{\kappa} -\lambda -\kappa^2 F(\lambda_{\kappa}^{\infty},\kappa) = \lambda_{\kappa}^{\infty} -\lambda -i\kappa^2 \Im F(\lambda_{\kappa}^{\infty},\kappa)\\
D_1(\lambda,\kappa)& = \lambda_{\kappa} -\lambda -\kappa^2 F(\lambda_{\kappa}^{\infty},\kappa)-\kappa^2 F'(\lambda_{\kappa}^{\infty},\kappa)(\lambda-\lambda_{\kappa}^{\infty})\\
&= \lambda_{\kappa}^{\infty} -\lambda -i\kappa^2 F(\lambda_{\kappa}^{\infty},\kappa)-\kappa^2 F'(\lambda_{\kappa}^{\infty},\kappa)(\lambda-\lambda_{\kappa}^{\infty})
\end{align*}
Note that by Hypothesis ({\bf H2}), $\Im \widehat{D}(\lambda, \kappa) = -\kappa^2 \Im F(\lambda_{\kappa}^{\infty},\kappa) <0$ as soon as $\kappa \neq 0$. This allows us to define the function $\widehat{G}$ for $\lambda \in [a_0,b_0]$, $0< |\kappa |\leq \kappa_1$ by $\widehat{G}(\lambda,\kappa)=\widehat{D}(\lambda,\kappa)^{-1}$.
For $\lambda \in [a_0,b_0]$ and $|\kappa |\leq \kappa_1$, we also have that:
\begin{align}\label{D1Dhat}
|D_1(\lambda,\kappa)| &\ge |\lambda_{\kappa}^{\infty} -\lambda -i\kappa^2 F(\lambda_{\kappa}^{\infty},\kappa)| -\kappa^2 \left(\sup_{(\lambda,\kappa)\in [a_0,b_0]\times [-\kappa_0,\kappa_0]}|F'(\lambda,\kappa)| \right) |\lambda_{\kappa}^{\infty} -\lambda | \nonumber\\
&\ge \left(1-\kappa^2 \sup_{(\lambda,\kappa)\in [a_0,b_0]\times [-\kappa_0,\kappa_0]} |F'(\lambda,\kappa)|\right) |\widehat{D}(\lambda,\kappa)|\ge \delta_1 |\widehat{D}(\lambda,\kappa)| \, .
\end{align}
In particular, for $\lambda \in [a_0,b_0]$ and $0< |\kappa |\leq \kappa_1$, we define the function $G_1$ by $G_1(\lambda,\kappa)=D_1(\lambda,\kappa)^{-1}$ and it holds:
$$
|G_1(\lambda,\kappa)| \le \delta_1^{-1} |\widehat{G}(\lambda,\kappa)| \, .
$$
Finally, for $(\lambda,\kappa) \in [a_0,b_0]\times [-\kappa_0,\kappa_0]$, $D(\lambda,\kappa)= \lambda_{\kappa}^{\infty} - \lambda -i\kappa^2 \Im F(\lambda_{\kappa}^{\infty},\kappa) -\kappa^2(F(\lambda,\kappa)-F(\lambda_{\kappa}^{\infty},\kappa))$. Hence, for $\lambda \in [a_0,b_0]$ and $|\kappa |\le \kappa_1$, we obtain via the Mean Value Theorem that:
\begin{align}\label{DDhat}
|D(\lambda,\kappa)| &\ge |\lambda_{\kappa}^{\infty} -\lambda -i\kappa^2 F(\lambda_{\kappa}^{\infty},\kappa)| -\kappa^2 \left(\sup_{(\lambda,\kappa)\in [a_0,b_0]\times [-\kappa_0,\kappa_0]}|F'(\lambda,\kappa)| \right) |\lambda_{\kappa}^{\infty} -\lambda | \nonumber\\
&\ge \left(1-\kappa^2 \sup_{(\lambda,\kappa)\in [a_0,b_0]\times [-\kappa_0,\kappa_0]} |F'(\lambda,\kappa)|\right) |\widehat{D}(\lambda,\kappa)|\ge \delta_1 |\widehat{D}(\lambda,\kappa)| \, .
\end{align}
In particular, for $\lambda \in [a_0,b_0]$ and $0< |\kappa |\leq \kappa_1$, it holds:
$$
|G(\lambda,\kappa)| \le \delta_1^{-1} |\widehat{G}(\lambda,\kappa)| \, .
$$

Recall that the function $g$ vanishes outside $[a,b]$ and $g\equiv 1$ on $[a_0,b_0]$. So, we can write for $t\in [0,\infty)$, $0<|\kappa |\leq \kappa_1$, ${\mathcal I}(t,\kappa) = {\mathcal I}_1(t,\kappa) + {\mathcal I}_2(t,\kappa) + {\mathcal I}_{\partial}(t,\kappa)$ where
\begin{align*}
{\mathcal I}_2(t,\kappa) &:= \frac{1}{\pi}\int_{a_0}^{b_0} d\lambda \, e^{-i\lambda t}\Im\,(G(\lambda, \kappa)-G_1(\lambda, \kappa)) \, ,\\
{\mathcal I}_{\partial}(t,\kappa) &:= \frac{1}{\pi}\int_a^{a_0} d\lambda \, e^{-i\lambda t}g(\lambda)\Im\,G(\lambda, \kappa) + \frac{1}{\pi}\int_{b_0}^b d\lambda \, e^{-i\lambda t}g(\lambda)\Im\, G(\lambda, \kappa) \, .
\end{align*}
It remains to perform the analysis of each term.

{\bf Step 1.} We start with the term ${\mathcal I}_1$. We write for $\lambda \in [a_0,b_0]$, $0< |\kappa |\leq \kappa_1,$
$$
\Im\, G_1 (\lambda,\kappa) =\frac{1}{2i}\left( \frac{1}{b_\kappa - a_\kappa \lambda}-\frac{1}{\overline{b_\kappa} -\overline{a_\kappa} \lambda}\right),
$$
where
$a_\kappa :=1+\kappa ^2 F'(\lambda_\kappa ^{\infty},\kappa)$ and $b_{\kappa} := \lambda_{\kappa}^{\infty} a_{\kappa} -i\kappa^2 \Im F_{\kappa} (\lambda_{\kappa}^{\infty})= \lambda_{\kappa}+\kappa^2 (\lambda_{\kappa}^{\infty} F'(\lambda_\kappa ^{\infty},\kappa)) -F(\lambda_\kappa ^{\infty},\kappa))$. Note that $\delta_1 \leq 1- \kappa^2 \sup_{(\lambda,\kappa) \in [a_0,b_0]\times [-\kappa_0,\kappa_0]} |F'(\lambda,\kappa)| \leq | a_{\kappa}|$ for any $|\kappa |\leq \kappa_1$. Now, for $0< |\kappa |\leq \kappa_1$, we consider the function $g_1(\cdot, \kappa)$ defined by
\begin{equation}\label{functionf}
g_1(z,\kappa)= \frac{1}{2i}\left(\frac{1}{b_\kappa -a_\kappa z}-\frac{1}{\overline{b_\kappa} -\overline{a_\kappa}z}\right),
\end{equation}
which is meromorphic in the complex plane, with poles at $\zeta_\kappa$ and $\overline{\zeta_\kappa}$:
$$
{\zeta_\kappa}= \lambda_\kappa ^{\infty} -i\kappa ^2 \, \frac{\Im F(\lambda_\kappa ^{\infty},\kappa)}{a_{\kappa}} \, .
$$
In particular, for $\lambda \in [a_0,b_0]$, $0< |\kappa |\leq \kappa_1$, $g_1(\lambda, \kappa) = \Im G_1(\lambda, \kappa)$. Note that for any $0< |\kappa |\leq \kappa_1,$
\begin{gather}\label{bounddist}
\left. \begin{array}{r} |\lambda_{\kappa} ^{\infty} - \Re \zeta_{\kappa}| \\
|\Im \zeta_{\kappa}|
\end{array} \right\} \leq \kappa^2 \, \delta_1^{-1} \Im F(\lambda_\kappa ^{\infty},\kappa) \leq \kappa^2 \, \delta_1^{-1} \sup_{(\lambda,\kappa) \in [a_0,b_0]\times [-\kappa_0,\kappa_0]} |F'(\lambda,\kappa)| \, .
\end{gather}
Explicit calculations also yield:
$$\Im \zeta_{\kappa} = -\kappa^2 (1+\kappa^2 \Re F'(\lambda_{\kappa} ^{\infty},\kappa)) \frac{\Im F(\lambda_{\kappa} ^{\infty},\kappa)}{|a_{\kappa}|^2}\, .$$
Once observed that for $0< |\kappa |\leq \kappa_1$, $\delta_1 \leq 1- \kappa^2 \sup_{(\lambda,\kappa) \in [a_0,b_0]\times [-\kappa_0,\kappa_0]} |F'(\lambda,\kappa)| \leq 1+\kappa^2 \Re F'(\lambda_{\kappa} ^{\infty},\kappa)$, we also deduce that:
\begin{equation*}
\Im \zeta_{\kappa} \leq -\kappa^2 \delta_1 \frac{\Im F(\lambda_{\kappa} ^{\infty},\kappa)}{(1+\kappa_1^2 \sup_{(\lambda,\kappa) \in [a_0,b_0]\times [-\kappa_0,\kappa_0]} |F'(\lambda,\kappa)|)^2} < 0 \, .
\end{equation*}
In other words, for any $0< |\kappa | \leq \kappa_1$, the pole $\zeta_\kappa$ lies in the lower half-plane and:
$$-\kappa^2 \, \delta_1^{-1} \Im F(\lambda_\kappa ^{\infty},\kappa) \leq \Im \zeta_{\kappa} \leq - \kappa^2 \delta_1 \frac{\Im F(\lambda_{\kappa} ^{\infty},\kappa)}{(1+\kappa_1^2 \sup_{(\lambda,\kappa) \in [a_0,b_0]\times [-\kappa_0,\kappa_0]} |F'(\lambda,\kappa)|)^2} < 0 \, .
$$

Now, fix $0< \delta_1' <\delta_1$ and $0< \kappa_1' \leq \kappa_1$ such that: $\kappa_1'^2 \, \sup_{(\lambda,\kappa) \in [a_0,b_0]\times [-\kappa_0,\kappa_0]} |F'(\lambda,\kappa)| \leq \delta_1' \delta_1$. In view of the bound (\ref{bounddist}) and the fact that $\lambda_{\kappa}^{\infty} \in [a_0 +\delta_1, b_0 -\delta_1]$ for any $0< |\kappa |\leq \kappa_1$, this implies that for any $0< |\kappa | \leq \kappa_1'$, $\Re \zeta_{\kappa} \in [a_0 +(\delta_1-\delta_1'), b_0 -(\delta_1-\delta_1')]$. Let $\gamma$ be a fixed smooth curve in the lower half plane, joining the endpoints of the interval $[a_0,b_0]$ and staying at positive distance from the closure of the bounded sets $\{\zeta_{\kappa}; 0< |\kappa |\leq \kappa_1' \}$ and $\{\overline{\zeta_{\kappa}}; 0< |\kappa |\leq \kappa_1' \}$. Then, for $0< |\kappa |\leq \kappa_1'$, the closed curve $J\cup\gamma$ enclose only the pole $\zeta_\kappa$ and so,
\begin{equation}\label{cauchy}
\frac{1}{\pi }\oint_{J\cup \gamma^-} e^{-izt} g_1(z,\kappa)\, dz =c_\kappa e^{-i\zeta_\kappa t}\, ,
\end{equation}
with $c_{\kappa} = a_{\kappa}^{-1}$. Therefore,
\begin{equation}\label{intG1}
{\mathcal I}_1(t,\kappa) = \frac{1}{\pi} \int_{a_0}^{b_0} e^{-i\lambda t} \Im\, G_1(\lambda,\kappa)\, d\lambda = c_\kappa \, e^{-i\zeta_\kappa t}+\frac{1}{\pi }\int_{\gamma} e^{-izt} g_1(z,\kappa)\, dz
\end{equation}
Now, for all $z\in \gamma$, $g_1(z,\kappa) =\kappa^2 h_1(z,\kappa)$ where
$$
h_1(z,\kappa)= \frac{p_{\kappa}z+q_{\kappa}}{|a_{\kappa}|^2(z-\zeta_\kappa)(z-\overline{\zeta_\kappa})}\, ,
$$
with $p_{\kappa}=\Im\, F'(\lambda_\kappa ^{\infty},\kappa)$ and $q_{\kappa}=\Im\, F(\lambda_\kappa ^{\infty},\kappa)-\lambda_\kappa ^{\infty}\Im\, F'(\lambda_\kappa ^{\infty},\kappa)$.

By construction, $\inf_{z\in \gamma, 0< |\kappa |\leq \kappa_1'} |z-\zeta_\kappa | >0$ and $\inf_{z\in \gamma, 0< |\kappa |\leq \kappa_1'} |z-\overline{\zeta_\kappa} | >0$, so the functions $h_1(\cdot ,\kappa)$ are analytic in some fixed open region containing $\gamma$ for any $0< |\kappa |\leq \kappa_1'$. Once combined with Hypotheses ({\bf H0})(b) and ({\bf H1})(a), this implies that $\sup_{z\in \gamma, 0< |\kappa |\leq \kappa_1'} |h_1(z,\kappa) |< \infty$ and $\sup_{z\in \gamma, 0< |\kappa |\leq \kappa_1'} |h_1'(z,\kappa) | <\infty$.

Now, note that for any $t\geq 0$ and any $z\in \gamma$, $| e^{-iz t}|\le 1$, since the curve $\gamma$ is contained in the lower half plane. We have that for any $t\geq 0$ and $0< |\kappa |\leq \kappa_1',$
\begin{align}
\left| \int_{\gamma} e^{-izt} g_1(z,\kappa)\, dz \right| & \leq C \kappa^2 \, , \label{intgamma}\\
\left| \int_{\gamma} e^{-izt} g_1'(z,\kappa)\, dz \right| & \leq C\kappa^2 \, .\label{tintgamma}
\end{align}
for some $C>0$.

{\bf Step 2.} In order to conclude, first define for $t\geq 0$ and $0< |\kappa |\leq \kappa_1',$
\begin{align}\label{RII}
R(t,\kappa) &:= {\mathcal I}(t,\kappa) - c_{\kappa} \, e^{-i\zeta_{\kappa} |t|} = ({\mathcal I}_1(t,\kappa) - c_{\kappa}\, e^{-i\zeta_{\kappa} |t|})+ {\mathcal I}_2(t,\kappa) + {\mathcal I}_{\partial}(t,\kappa)\\
&= \frac{1}{\pi }\int_{\gamma} e^{-izt} g_1(z,\kappa)\, dz + {\mathcal I}_2(t,\kappa) + {\mathcal I}_{\partial}(t,\kappa)\nonumber \, ,
\end{align}
due to (\ref{intG1}). According to (\ref{intgamma}), Corollary \ref{error1} and Proposition \ref{bords}, all the terms on the RHS are of order $\kappa^2$ (for $\kappa$ small enough), which yields our first estimate on $R$.

Now, note that integration by parts yields for $t\geq 0$ and $0< |\kappa |\leq \kappa_1',$
\begin{align*}
\int_{\gamma}  e^{-iz t} g_1'(z,\kappa)\, dz &= it\int_{\gamma} e^{-iz t} g_1(z,\kappa)\, dz + e^{-ib_0 t} \Im G_1(b_0,\kappa) - e^{-ia_0 t} \Im G_1(a_0,\kappa) \, , \\
\int_{a_0}^{b_0} d\lambda \, e^{-i\lambda t}\Im\,(G'(\lambda, \kappa)-G_1'(\lambda, \kappa)) &= it \int_{a_0}^{b_0} d\lambda \, e^{-i\lambda t}\Im\,(G(\lambda, \kappa)-G_1(\lambda, \kappa)) \\
&+ e^{-ib_0 t} \Im\,(G(b_0, \kappa)-G_1(b_0, \kappa)) \\
&- e^{-ia_0 t}\Im\,(G(a_0, \kappa)-G_1(a_0, \kappa))\\
\int_a^{a_0} d\lambda \, e^{-i\lambda t} (g(\lambda)\Im\,G(\lambda, \kappa))'  &= it \int_a^{a_0} d\lambda \, e^{-i\lambda t}g(\lambda)\Im\,G(\lambda, \kappa) + e^{-ia_0 t}\Im\,G(a_0, \kappa) \, , \\
\int_{b_0}^b d\lambda \, e^{-i\lambda t} (g(\lambda)\Im\, G(\lambda, \kappa))'  &= it \int_{b_0}^b d\lambda \, e^{-i\lambda t}g(\lambda)\Im\, G(\lambda, \kappa) - e^{-ib_0 t}\Im\,G(b_0, \kappa) \, ,
\end{align*}
where we have used $g(a)=0=g(b)$ and $g(a_0)=1=g(b_0)$. It follows from (\ref{RII}) that for $t\geq 0$ and $0< |\kappa |\leq \kappa_1',$
\begin{equation}\label{RJJ}
it R(t,\kappa) = {\mathcal J}_1(t,\kappa) + {\mathcal J}_2(t,\kappa) + {\mathcal J}_{\partial}(t,\kappa) \, ,
\end{equation}
where
\begin{align*}
{\mathcal J}_1(t,\kappa) &= \frac{1}{\pi } \int_{\gamma}  e^{-iz t} g_1'(z,\kappa)\, dz\\
{\mathcal J}_2(t,\kappa) &= \frac{1}{\pi } \int_{a_0}^{b_0} d\lambda \, e^{-i\lambda t}\Im\,(G'(\lambda, \kappa)-G_1'(\lambda, \kappa)) \\
{\mathcal J}_{\partial}(t,\kappa) &= \frac{1}{\pi} \left(\int_a^{a_0} d\lambda \, e^{-i\lambda t} (g(\lambda)\Im\,G(\lambda, \kappa))' + \int_{b_0}^b d\lambda \, e^{-i\lambda t} (g(\lambda)\Im\, G(\lambda, \kappa))' \right) \, .
\end{align*}
According to (\ref{tintgamma}) and Proposition \ref{bords}, the first and third terms on the RHS of (\ref{RJJ}) are of order $\kappa^2$. In view of Corollary \ref{error2}, the second one is of order $\kappa^{2\alpha}$ (resp. $\kappa^2 |\log |\kappa ||$) if $\alpha \in (0,1)$ (resp. if $\alpha=1$). This completes the proof of statement (b).

The last part of Theorem \ref{protothm} is a direct consequence of formula (\ref{zetax}).

%*********************************
\subsection{Technicalities}

In this section, the results are stated under Hypotheses ({\bf H0}), ({\bf H1}) and ({\bf H2}). The quantities $\delta_1 \in (0, \min (|\lambda_0 -a_0 |, |b_0- \lambda_0 |,1))$ and $0 <\kappa_1 \leq \kappa_0$ are fixed according to conditions explicited in the proof of Theorem \ref{protothm}.

First, we provide upper bounds on the terms ${\mathcal I}_{\partial}$ and ${\mathcal J}_{\partial}$:
\begin{pro}\label{bords} There exists $C>0$ such that for all $t\in {\mathbb R}$, $0< |\kappa |\leq \kappa_1$, $|{\mathcal I}_{\partial} (t,\kappa) | \leq C \kappa^2$ and $|{\mathcal J}_{\partial} (t,\kappa) | \leq C \kappa^2$.
\end{pro}
\noindent{\bf Proof:} We deduce from Lemma \ref{fixed} that for all $\lambda \in [a,a_0]$, $|\kappa |\leq \kappa_1$, $|D(\lambda,\kappa)| \geq a_0 +\delta_1 -\lambda \geq \delta_1 >0$. On the other hand, for $\lambda \in [a,a_0]$, $0< |\kappa |\leq \kappa_1,$
\begin{equation}\label{ImG}
\Im\,G(\lambda, \kappa) = \kappa^2 \frac{\Im F(\lambda,\kappa)}{|D(\lambda,\kappa) |^2}
\end{equation}
In view of ({\bf H0})(b), we deduce that for $t\in {\mathbb R}$, $0< |\kappa |\leq \kappa_1,$
\begin{equation*}
\left| \int_a^{a_0} \, d\lambda \, e^{-i\lambda t}g(\lambda)\Im\,G(\lambda, \kappa) \right| \leq \kappa^2 \sup_{\lambda \in [a,a_0]}|\Im F(\lambda,\kappa)| \int_a^{a_0} \frac{d\lambda}{(a_0 +\delta_1 -\lambda)^2} \leq C \kappa^2 \, .
%\left| \int_{b_0}^b \, d\lambda \, e^{-i\lambda t}g(\lambda)\Im\, G(\lambda, \kappa) \right| &\leq \kappa^2 \sup_{\lambda \in [b_0,b]}|\Im F(\lambda,\kappa)| \int_{b_0}^b \,\,\, \frac{d\lambda}{(\lambda-b_0 +\delta_1)^2} \, .
\end{equation*}
For all $\lambda \in [a,a_0]$, $0< |\kappa |\leq \kappa_1$, $(g \Im\,G)'= g' \Im\,G+g \Im\,G'$ and we deduce from (\ref{ImG}) that
\begin{equation*}
\Im\,G'(\lambda, \kappa) = \kappa^2 \frac{\Im F'(\lambda,\kappa)}{|D(\lambda,\kappa) |^2} - 2\kappa^2 \frac{\Re (\overline{D(\lambda, \kappa)} D'(\lambda, \kappa)) \Im F(\lambda,\kappa)}{|D(\lambda,\kappa) |^4}
\end{equation*}
with $D'(\lambda, \kappa) = -1- \kappa^2 F'(\lambda,\kappa)$. It follows that for $t\in {\mathbb R}$, $0< |\kappa |\leq \kappa_1,$
\begin{align*}
\left| \int_a^{a_0} \, d\lambda \, e^{-i\lambda t} (g(\lambda)\Im\,G(\lambda, \kappa))' \right| &\leq \kappa^2 \int_a^{a_0} \, d\lambda \, \frac{g (\lambda ) |\Im F'(\lambda,\kappa)| + |g' (\lambda)| |\Im F(\lambda,\kappa)|}{(a_0 + \delta_1 -\lambda)^2} \\
&+ 2 \kappa^2 \int_a^{a_0} \, d\lambda \, g (\lambda)\, \frac{(1+ \kappa^2 |F'(\lambda,\kappa)|) |\Im F(\lambda,\kappa)|}{(a_0 +\delta_1 -\lambda)^3} \\
&\leq C \kappa^2 \, ,
\end{align*}
in view of Hypotheses ({\bf H0})(b) and ({\bf H1})(a).
%Similarly, for $t\in {\mathbb R}$, $0< |\kappa |\leq \kappa_1,$
%\begin{align*}
%\left| \int_{b_0}^b \, d\lambda \, e^{-i\lambda t}(g(\lambda)\Im\, G(\lambda, \kappa))' \right| &\leq \kappa^2 \int_{b_0}^b \, d\lambda \, \frac{g (\lambda ) |\Im F'(\lambda,\kappa)| + |g' (\lambda)| |\Im F(\lambda,\kappa)|}{(\lambda-b_0 +\delta_1)^2} \\
%&+ 2 \kappa^2 \int_{b_0}^b \, d\lambda \, g (\lambda)\, \frac{(1+ \kappa^2 |F'(\lambda,\kappa)|) |\Im F(\lambda,\kappa)|}{(\lambda-b_0+\delta_1)^3} \, .
%\end{align*}
A similar procedure applies to the term
$$\int_{b_0}^b \, d\lambda \, e^{-i\lambda t}g(\lambda)\Im\,G(\lambda, \kappa)$$ and the conclusion of the proposition follows. \ep

Now, we provide some upper bounds on the terms ${\mathcal I}_2$ and ${\mathcal J}_2$, which rely on the following lemma:
\begin{lemma}\label{cases} Let $(\alpha, \beta) \in [0,\infty)^2$, $0< |\kappa | <\kappa_1$ and $z_\kappa = \lambda_{\kappa}^{\infty} -i\kappa^2 \Im\,  F(\lambda_{\kappa}^{\infty},\kappa) = \lambda_{\kappa} -\kappa^2 F(\lambda_{\kappa}^{\infty},\kappa)$. There exist $C>0$ and $0<\kappa_2 \leq \kappa_1$, such that for any $0< |\kappa | \leq \kappa_2,$
$$
\int_{a_0}^{b_0} \, \frac{|\lambda-\Re z_{\kappa}|^{\alpha}}{|\lambda -z_{\kappa}|^{\beta}}\, d\lambda \leq \left\{
\begin{array}{lcc}
C \kappa^{2(\alpha-\beta+1)} & \text{if} & \alpha-\beta+1 <0\\
C |\log |\kappa || & \text{if} & \alpha-\beta+1 =0\\
C & \text{if} & \alpha-\beta+1 >0
\end{array} \right.
$$
\end{lemma}

\begin{proof} We start with some preliminary remarks. By Lemma \ref{fixed}, $\Re z_{\kappa} =\lambda_{\kappa}^{\infty} \in [a_0+\delta_1,b_0-\delta_1]$ for any $0< |\kappa | \leq \kappa_1$. Lemma \ref{fixed} and Hypothesis ({\bf H0})(a) also imply: $\lim_{\kappa \rightarrow 0} \lambda_{\kappa}^{\infty} =\lambda_0$. Finally, $\Im z_{\kappa} = -i\kappa^2 \Im\,  F(\lambda_{\kappa}^{\infty},\kappa) <0$ due to ({\bf H2}). Since $F$ is continuous at $(\lambda_0,0)$, we obtain that:
$$ \lim_{\kappa \rightarrow 0} \frac{\Im z_{\kappa}}{\kappa^2} = -\Im F(\lambda_0,0) < 0 \, .$$
So, given $0< \delta_2 < \Im F(\lambda_0,0)$, we can fix $0< \kappa_2 \leq \kappa_1$, such that for any $0< |\kappa | <\kappa_2,$
\begin{equation}\label{k3}
- \kappa^2 \sup_{(\lambda,\kappa)\in [a_0,b_0]\times [-\kappa_0,\kappa_0]} |F(\lambda,\kappa)| \leq \Im z_{\kappa} \leq - \kappa^2 \delta_2 \, .
\end{equation}
Now, with the change of variables $\lambda - \Re z_{\kappa}= \mu |\Im z_{\kappa} |$, we obtain that:
$$
\int_{a_0}^{b_0} \, \frac{|\lambda-\Re z_{\kappa}|^{\alpha}}{|\lambda -z_{\kappa}|^{\beta}}\, d\lambda=
|\Im z_{\kappa} |^{\alpha-\beta+1} \int_{a_{\kappa}}^{b_{\kappa}}\frac{|\mu |^{\alpha}}{(\mu^2 +1)^{\frac{\beta}{2}}}\, d\mu
$$
where
\begin{gather*}
a_{\kappa} = \frac{\Re z_{\kappa} -a_0}{|\Im z_{\kappa} |} \quad \text{and} \quad b_{\kappa} = \frac{\Re z_{\kappa} -b_0}{|\Im z_{\kappa} |}
\end{gather*}
We denote $\gamma=\alpha -\beta$.

\noindent{\bf Case $\gamma\neq -1 $.} We split the integral, integrating on the intervals $[a_{\kappa},\max (-1,a_{\kappa})]$, $[\max (-1,a_{\kappa}),\min (1,b_{\kappa})]$ and $[\min (1,b_{\kappa}),b_{\kappa}]$. On the interval $[\max (-1,a_{\kappa}),\min (1,b_{\kappa})]$, the integral is bounded by the same integral on $[-1,1]$ for which we observe that the integrand is bounded by $|\mu|^{\alpha}\le 1$. This term is finally bounded by $2 |\Im z_{\kappa} |^{\gamma+1}$. The integral on $[\min (1,b_{\kappa}),b_{\kappa}]$ is bounded by:
$$
|\Im z_{\kappa} |^{\gamma+1} \int_{1}^{b_\kappa}\frac{|\mu|^{\alpha}}{(\mu^2 +1)^{\frac{\beta}{2}}}d\mu \le
|\Im z_{\kappa} |^{\gamma+1} \int_{1}^{b_\kappa}\mu^{\gamma}d\mu =\frac{(b_0 - \Re z_{\kappa})^{\gamma+1}}{(\gamma +1)}-\frac{|\Im z_{\kappa} |^{\gamma+1}}{\gamma+1}\, .
$$
We manage the integral on the interval $[a_{\kappa},\max (-1,a_{\kappa})]$ analogously. Estimates for the case $\gamma\neq -1$ follows now from (\ref{k3}).

\noindent{\bf Case $\gamma =-1$.} We split again the integral, integrating on the intervals $[a_{\kappa},\max (-1,a_{\kappa})]$, $[\max (-1,a_{\kappa}),\min (1,b_{\kappa})]$ and $[\min (1,b_{\kappa}),b_{\kappa}]$. On the interval $[\max (-1,a_{\kappa}),\min (1,b_{\kappa})]$, the integral is bounded by the same integral on $[-1,1]$, which is bounded by $2$. The integral on $[\min (1,b_{\kappa}),b_{\kappa}]$ is bounded by:
$$
\int_{1}^{b_\kappa}\frac{|\mu|^{\alpha}}{(\mu^2 +1)^{\frac{\beta}{2}}}\, d\mu \le \int_{1}^{b_\kappa}\frac{1}{\mu}\, d\mu
=\ln (b_0 - \Re z_{\kappa}) - \ln |\Im z_{\kappa} | \, .
$$
We manage the integral on the interval $[a_{\kappa},\max (-1,a_{\kappa})]$ analogously. Estimates for the case $\gamma = -1$ follows again from (\ref{k3}).
\end{proof}

\begin{lemma}\label{g-g1} There exists $C>0$ such that for any $\lambda \in [a_0,b_0]$, $0< |\kappa |\leq \kappa_1,$
\begin{align*}
|G(\lambda,\kappa)-G_1 (\lambda,\kappa)| &\le C \kappa^2 |\widehat{G}(\lambda,\kappa)| \sup_{(\lambda,\kappa)\in [a_0,b_0] \times [-\kappa_0,\kappa_0]} |F'(\lambda,\kappa)| \, ,\\
|G(\lambda,\kappa)-G_1 (\lambda,\kappa)| &\le C \kappa^2 |\widehat{G}(\lambda,\kappa)|^2 |\lambda-\lambda_{\kappa}^{\infty}|^{\alpha+1} \, .
\end{align*}
\end{lemma}
\begin{proof} For $\lambda \in [a_0,b_0]$, $0< |\kappa |\leq \kappa_1$, we note that $G_1-G=G(D-D_1)G_1$ and so $|G-G_1|\le \delta_1^{-2} |\widehat{G}|^2 |D-D_1|$ (see e.g. (\ref{D1Dhat}) and (\ref{DDhat})). On the other hand, for $\lambda \in [a_0,b_0]$, $|\kappa |\leq \kappa_1,$
\begin{align}\label{D-D1}
D(\lambda,\kappa) -D_1 (\lambda,\kappa) &=\kappa^2 \left(F(\lambda,\kappa)-F(\lambda_{\kappa}^{\infty},\kappa)-F'(\lambda_{\kappa}^{\infty},\kappa)(\lambda-\lambda_{\kappa}^{\infty}) \right)\\
&= \kappa^2 \int_{\lambda_{\kappa}^{\infty}}^\lambda (F'(\mu,\kappa)-F'(\lambda_{\kappa}^{\infty},\kappa))\, d\mu\, . \nonumber
\end{align}
({\bf H1})(a) entails: $|D(\lambda,\kappa) -D_1(\lambda,\kappa)|\le 2 C\kappa^2 |\lambda - \lambda_{\kappa}^{\infty}| \sup_{(\lambda,\kappa)\in [a_0,b_0] \times [-\kappa_0,\kappa_0]} |F'(\lambda,\kappa)|$, which leads to the first estimate once noted that $|\widehat{G}(\lambda,\kappa)| |\lambda - \lambda_{\kappa}^{\infty}| \leq 1$ for $\lambda \in [a_0,b_0]$ and $0< |\kappa |\leq \kappa_1$. ({\bf H1})(b) entails: $|D(\lambda,\kappa) -D_1(\lambda,\kappa)|\le C\kappa^2 |\lambda-\lambda_{\kappa}^{\infty}|^{\alpha +1}$, which leads to the second estimate.
\end{proof}

\begin{corollary}\label{error1} There exists $C>0$ such that for any $t\in {\mathbb R}$, $0< |\kappa |\leq \kappa_2,$
$$
\left| {\mathcal I}_2(t,\kappa) \right| \le \int_{a_0}^{b_0}\, |G(\lambda,\kappa) -G_1 (\lambda,\kappa) |\, d\lambda \le C\kappa ^2 \, .
$$
\end{corollary}
\begin{proof} We integrate the second statement of Lemma \ref{g-g1} for $0< |\kappa |\leq \kappa_2$ and obtain that,
$$
\int_{a_0}^{b_0} |G(\lambda,\kappa)-G_1(\lambda,\kappa) | \, d\lambda \le C \kappa^2 \int_{a_0}^{b_0} \frac{|\lambda-\lambda_{\kappa}^{\infty}|^{\alpha +1}}{|\lambda - \lambda_{\kappa}^{\infty}+i\kappa^2 \Im\,  F(\lambda_{\kappa}^{\infty},\kappa)|^2} d\lambda \, .
$$
We use Lemma \ref{cases}, with $\beta =2$ and $\alpha +1$ instead of $\alpha$. Then, the number $\alpha -\beta +1$ in Lemma \ref{cases} is just $\alpha$, which is positive. The first and second statements follow.
\end{proof}

\begin{lemma}\label{g'-g1'} There exists $C>0$ such that for any $\lambda \in [a_0,b_0]$, $0< |\kappa |\leq \kappa_1,$
$$
|G'(\lambda,\kappa)-G_1'(\lambda,\kappa)|\le C \kappa^2|\widehat{G}(\lambda,\kappa)|^2 |\lambda-\lambda_{\kappa}^{\infty}|^{\alpha} \, .
$$
\end{lemma}
\begin{proof} For $\lambda \in [a_0,b_0]$, $0< |\kappa |\leq \kappa_1$, we can write the difference between the difference of the derivatives of $G$ and $G_1$ w.r.t. $\lambda$: $G'-G{_1}'=L_1+L_2+L_3$ where $L_1 =G(G_1-G)(1+\kappa^2 F'(\lambda,\kappa))$, $L_2 =\kappa^2GG_1(F'(\lambda_{\kappa}^{\infty},\kappa)-F'(\lambda,\kappa))$ and $L_3 =G(G_1-G)(1+\kappa^2 F'(\lambda_{\kappa}^{\infty},\kappa))$. Recall that for $\lambda \in [a_0,b_0]$, $0< |\kappa |\leq \kappa_1,$
\begin{itemize}
\item $G_1-G=G(D-D_1)G_1$ and so $|G-G_1|\le \delta_1^{-2} |\widehat{G}|^2 |D-D_1|$ (see (\ref{D1Dhat}) and (\ref{DDhat})),
\item $|D(\lambda,\kappa) -D_1(\lambda,\kappa)|\le C\kappa^2 |\lambda-\lambda_{\kappa}^{\infty}|^{\alpha +1}$ due to Hypothesis ({\bf H1})(b) and (\ref{D-D1}).
\end{itemize}
Up to some positive multiplicative constant, the quantities $L_1$ and $L_3$ can be bounded by $\kappa^2 |\widehat{G}(\lambda,\kappa)|^3 |\lambda -\lambda_{\kappa}^{\infty}|^{\alpha+1}$, while the term $L_2$ is bounded by $\kappa^2 | \widehat{G}(\lambda,\kappa)|^2 |\lambda -\lambda_{\kappa}^{\infty}|^{\alpha}$. The proof follows from the fact that $|\widehat{G}(\lambda,\kappa) |\, |\lambda -\lambda_{\kappa}^{\infty}|\le 1$ for $\lambda \in [a_0,b_0]$ and $0< |\kappa |\leq \kappa_1$.
\end{proof}

\begin{corollary}\label{error2} There exists $C>0$ such that for any $t\in {\mathbb R}$, $0< |\kappa |\leq \kappa_2,$
$$
\left| {\mathcal J}_2(t,\kappa) \right| \leq \int_{a_0}^{b_0} \, |G'(\lambda,\kappa) -G_1'(\lambda,\kappa)|\, d\lambda \leq C \left\{ \begin{array}{lcc}
\kappa^{2\alpha} & \text{if} & \alpha \in (0,1) \\
\kappa^2 |\log |\kappa || & \text{if} & \alpha =1\\
\kappa ^2 & \text{if} & \alpha >1
\end{array} \right. \, .
$$
\end{corollary}
\begin{proof} We integrate Lemma \ref{g'-g1'} for $0< |\kappa |\leq \kappa_2$ and obtain that:
$$
\int_{a_0}^{b_0} |G'(\lambda,\kappa)-G_1'(\lambda,\kappa) | \, d\lambda \le C \kappa^2 \int_{a_0}^{b_0} \frac{|\lambda-\lambda_{\kappa}^{\infty}|^{\alpha}}{|\lambda - \lambda_{\kappa}^{\infty}+i\kappa^2 \Im\,  F(\lambda_{\kappa}^{\infty},\kappa)|^2} d\lambda \, .
$$
Then, we apply Lemma \ref{cases} with $\beta =2$. Indeed, if $\alpha \in (0,1)$ (resp. $\alpha=1$, resp. $\alpha >1$), $\alpha -\beta +1= \alpha -1<0$ (resp. $= 0$, resp. $>0$), which proves the first statement. The second statement follows now from the first and Corollary \ref{error1}.
\end{proof}

%\begin{proof}
%\textcolor{red}{It follows directly from Theorem \ref{protothm1} proven in Section 6, taking $F(\lambda ,\kappa)=-\overline{F_{\kappa}(\lambda )}$. That is our $F(\lambda ,\kappa)$ in theorem \ref{protothm} differs from
%$F_{\kappa}(\lambda )$ of  theorem \ref{protothm1}  only in the sign of the real part.}
%\end{proof}

%%%%%%%%%%%%%%%%%%

\section{Rank one perturbations}

In this section, we shall prove of results concerning rank one perturbations of self--adjoint operators. In 
\ref{behavior} we show how the positivity on the imaginary part of the unperturbed reduced resolvent implies pure absolutely continuous spectrum. In \ref{res} we describe how a simple embedded eigenvalue turns into a resonance. Here  smoothness of the resolvent is required. In \ref{so} we relate this to dynamic behavior of the system.

%************************************
\subsection{Behavior of Spectra}\label{behavior}

Let $H_0$ be  a self-adjoint operator on Hilbert space  $\mathcal{H}$, $\psi \in \mathcal{H}$ a normalized vector and define
\begin{equation}\label{rankone1}
H_{\kappa} = H_0 + \kappa | \psi \rangle \langle \psi | \, , \quad \kappa \in \mathbb{R}\, .
\end{equation}
Let
$$
\mathcal{H}_{\psi}:=\overline{span\{ (H_{\kappa}-z)^{-1} \psi/ z \notin \mathbb R \}}
$$
be the cyclic subspace generated by $\psi$. This space is independent of $\kappa$ and reduces the operator $H_{\kappa}$,
for every $\kappa$, see \cite{don}.

Let us denote by $H_{\kappa}^{\psi}$ the part of $H_{\kappa}$ on $\mathcal{H}_{\psi}$,
i.e. $H_{\kappa}^{\psi}:\mathcal{H}_{\psi}\to\mathcal{H}_{\psi}$ is given by
$$
H_{\kappa}^{\psi}\gamma = H_{\kappa} \gamma,\,\,\, \mbox{for all } \gamma \in Dom H_{\kappa}\cap  \mathcal{H}_{\psi} \, .
$$
%
%If $Q_{\psi}$ is the orthogonal projection on $\mathcal{H}_{\psi}$, then we have,
%$$ Q_{\psi}  H_{\kappa} \subset H_{\kappa} Q_{\psi}. $$
%This means that the range $Q_{\psi}$ reduces $H_{\kappa}$ (see \cite{kato} p. 278).
Let $P$ be an orthogonal projection such that $PH_0 \subset H_0 P$, that is the range of $P$ reduces $H_0$
(see \cite{kato} p. 278) and denote $P^{\perp } := I-P$, $H_{\kappa}^{\perp}:=P^{\perp }  H_0 P^{\perp }$.
For $\kappa \in \RR$ then define
\begin{equation}\label{functionF}
\mathcal F_{\kappa} (z):= \langle \psi, P^{\perp } (H_{\kappa}^{\perp}-z)^{-1} P^{\perp }\psi \rangle
\end{equation}
and
\[
\mathcal F_{\kappa} (\lambda+i0):= \lim_{ \epsilon  \downarrow 0} \mathcal F_{\kappa} (\lambda +i \epsilon)
\]
Recall that  $\sigma_{ac}(H)$, $\sigma_{s}(H)$  and  $\sigma_{p}(H)$ denote the absolutely and singular and point  spectrum respectively. With the definitions introduced  above, our results on the behavior of the spectra read as follows,
\begin{theorem}\label{absence}
Let $J\subset \RR$ be an  open interval.  If for every $\lambda \in J$
\begin{equation}\label{fermi1}
\Im \mathcal F_0(\lambda +i0) > 0
\end{equation}
then $J\subset \sigma_{ac}(H_{\kappa}^{\psi})$ and $J\cap \sigma_{s}(H_{\kappa}^{\psi}) = \emptyset$,
for all $\kappa \neq 0$.
\end{theorem}

\begin{theorem}\label{pp}
Let $J \subset \RR$ be an open interval. Suppose that
\begin{enumerate}
\item
$\Im \mathcal F_0(\lambda +i 0)>0$ for every $\lambda \in J$.
% where $P$ in the definition of $\mathcal F _{\kappa}$ (see \eqref{functionF}) is any orthogonal projection which reduces $H_0$.
\item
If  $H_0 \varphi =\lambda \varphi $, for  some $\lambda \in J$, then $\langle \varphi, \psi \rangle \neq 0 $ and $\lambda$ is a simple eigenvalue.
\end{enumerate}
Then $J \subset \sigma_{ac}(H_\kappa)$ and $J \cap \sigma_p (H_\kappa)=\emptyset$, for all $\kappa \neq 0$.
\end{theorem}

\begin{remark} It could happen that $\sigma_{sc}(H_\kappa) \neq \phi$.
\end{remark}
Before going into the proofs, we shall need a preliminary result. Let us recall that
 $\rho(H) :=\mathbb C \setminus\sigma (H)$, where  $\sigma (H)$ is the spectrum of $H$.
 \begin{lemma}\label{reduced}
Let  $H$ be a self-adjoint operator and $P$ an orthogonal projection such that $PH\subset HP$, that is, $P \mathcal H$ reduces $H$. Then
for all $z\in \rho(H)$,
$$
P^{\perp}(H-z)^{-1} P^{\perp}= P^{\perp}(P^{\perp}HP^{\perp}-z)^{-1} P^{\perp}
$$
\end{lemma}

\begin{proof} Since $H$ is self--adjoint and the range of $P^\perp$ reduces $H$, we have that $H\vert _{ P^\perp \mathcal H}$ is also self--adjoint. By the basic criterium
for self--adjointness Ran$(H\vert _{ P^\perp \mathcal H}-z) = P^\perp \mathcal H$ (see \cite{RSV1} p. 256).

Let $\psi \in P^\perp \mathcal H$ be arbitrary. Then, there exists $\varphi \in DomH\vert _{ P^\perp \mathcal H} = P^{\perp}\mathcal H \cap Dom H$ such that $\psi =(H-z)P^{\perp}\varphi$. Now, $P^\perp (H-z)^{-1} P^\perp \psi =P^\perp (H-z)^{-1} P^\perp (H-z)P^\perp \varphi = P^\perp \varphi$ and
\begin{eqnarray*}
P^\perp ( P^\perp H P^\perp -z)^{-1} P^\perp \psi &=&P^\perp ( P^\perp H P^\perp -z)^{-1} P^\perp (H-z)P^\perp \varphi \\
&=& P^\perp ( P^\perp H P^\perp -z)^{-1}( P^\perp H P^\perp -z)P^\perp \varphi \\
&=& P^\perp \varphi.
\end{eqnarray*}
\end{proof}

 Now, let us prove Theorems \ref{absence} and \ref{pp}.

\begin{proof}[Proof of Theorem \ref{absence}]
Since $P$ and $P^{\perp}$ reduce $H_0$, we have that for any $z \in \mathbb C \setminus\RR$, $P(H_0 -z)^{-1}=(H_0 -z)^{-1}  P$, and the corresponding commutation relation with $P^{\perp}$ (see \cite{kato} Theorem.6.5, p.173, Ch.3.). Therefore,
 \begin{eqnarray*}
 \Im \langle \psi,  (H_0-\lambda -i \epsilon)^{-1}\psi \rangle & = &   \Im \langle P \psi ,  (H_0-\lambda -i \epsilon)^{-1}P \psi \rangle  \\
& + & \Im \langle P^{\perp}\psi,  (H_0-\lambda -i \epsilon)^{-1}P^{\perp}\psi \rangle \\
   & = &   \Im \langle P \psi ,  (H_0-\lambda -i \epsilon)^{-1}P \psi \rangle +\Im \mathcal F_0(\lambda +i\epsilon)  ,
\end{eqnarray*}
where for the last equality we  use Lemma \ref{reduced}.

Next we note that the term $ \Im \langle P \psi ,  (H_0-\lambda -i \epsilon)^{-1}P \psi \rangle $ is always nonnegative. Actually, for any self-adjoint operator $H$,
setting $u=(H-z)^{-1}\gamma$ we have
\begin{eqnarray*}
\Im \langle \gamma, (H-z)^{-1}\gamma \rangle &=&\Im \langle (H-z)u,  u \rangle \\
&=& \Im z \|u\|^2.
\end{eqnarray*}
Hence, $\Im \langle \gamma, (H-z)^{-1}\gamma \rangle$  is nonnegative when $\Im z >0$.  Holomorphic functions  which map the upper half plane into itself called Herglotz, Nevanlinna or Pick functions. The limits of this functions 
when approaching the real axis from above exists a.e., see \cite{simon1} Thm I.4.

Using hypothesis (\ref{fermi1}),  we deduce that for every $\lambda \in J$,
\[
\lim_{\epsilon \downarrow 0}  \Im \langle \psi,  (H_0-\lambda -i \epsilon)^{-1}\psi \rangle \ge  \Im \mathcal F_0(\lambda +i0) > 0.
\]
Now, by the spectral theorem, there exists a measure $\mu_{\psi}$ such that
$$
\langle \psi , (H_0-z)^{-1}\psi \rangle = \int _{\RR} \frac{d\mu_{\psi}(t)}{t-z}
$$
and $H_0$, restricted to the the subspace of cyclicity generated by $\psi$, that is $H_0^\psi$  defined above, is unitarily equivalent to multiplication by the identity in $L_2(\RR, d\mu_{\psi})$.

A result essentially due to Aronszajn and Donoghue (see \cite{don}, Theorem 2 and also \cite{simon1} Theorem II.2 (iii)) states that the absolutely continuous part of $\mu_{\psi}$
is given by $\mu_{\psi ac}(\Delta) = \mu_{\psi}(\Delta  \cap A)$, where
$$
A=\{ \lambda |  0<  \Im \langle \psi,  (H_0-\lambda -i 0)^{-1}\psi \rangle < \infty \} .
$$

Also, the singular part of $\mu_{\psi}$
is given by $\mu_{\psi s}(\Delta) = \mu_{\psi}(\Delta \cap B)$, where
$$
B=\{ \lambda  |   \Im \langle \psi,  (H_0-\lambda -i 0)^{-1}\psi \rangle = \infty \} ,
$$
for any Borel set $\Delta$.

In fact due to Krein's formula (see I.13 in \cite{simon1}), we have that:
\begin{itemize}
\item[a)] $A$ is the support of the a.c. part of the measure $\mu_{\psi}^{\kappa}$ in the expression
$$
\langle \psi,  (H_\kappa-z)^{-1}\psi \rangle = \int _{\RR} \frac{d\mu_{\psi}^{\kappa}(t)}{t-z},
$$
for all $\kappa$ (see \cite{don}).
\item[b)] If $ \Im \langle \psi,  (H_0-\lambda -i 0)^{-1}\psi \rangle = \infty$ then $ \Im \langle \psi,  (H_\kappa-\lambda -i 0)^{-1}\psi \rangle = 0$ for $\kappa \neq 0$.
\end{itemize}
Therefore the result follows.
\end{proof}

\begin{proof}[Proof of Theorem \ref{pp}] We have the decomposition $\mathcal H =\mathcal H_\psi \oplus \mathcal H_ \psi^\perp$, where $ \mathcal H_\psi$ is the cyclic subspace generated by $\psi$. This subspace
reduces $H_\kappa$, for all $\kappa$ (\cite{don}). Let us denote as above by $H_\kappa^\psi$ the part of $H_\kappa $ on $\mathcal H_\psi$ and by ${H_\kappa^\psi}^\perp$ the
part of $H_\kappa $ on $\mathcal H_ \psi^\perp$ . Then,
$$
H_\kappa^\psi : \mathcal H_ \psi \to \mathcal H_ \psi
$$
$$
{H_\kappa^\psi}^\perp : \mathcal H _\psi^\perp \to \mathcal H _\psi^\perp
$$

Since by Theorem \ref{absence} $J \subset \sigma_{ac}(H_\kappa^\psi ) \subset\sigma (H_\kappa)$, we conclude $J \subset \sigma_{ac}(H_\kappa )$ .

Now assume $H_\kappa \theta = \lambda \theta$, with $\lambda \in J$. Write $\theta=\theta_1+\theta_2$, where $\theta_1\in \mathcal H_\psi$,
 $\theta_2\in \mathcal H_\psi^\perp$, $H_\kappa \theta_1=\lambda\theta_1$ and  $H_\kappa \theta_2=\lambda\theta_2$ ($\mathcal H_\psi$ and $\mathcal H_\psi ^\perp$
 reduce $H_\kappa$, for all $\kappa$).

 Since $H_\kappa^\psi$ has pure a.c. spectrum in $J$, by Theorem \ref{absence} it follows that $\theta_1=0$. Since $\theta _2 \in \mathcal H_\psi^\perp$, we
 deduce that $\theta _2 \perp \psi$ and $\langle \theta,  \psi \rangle =0$. Hence,
 $$
 \lambda \theta =H_\kappa \theta = H_0\theta.
 $$
 Therefore, $\lambda$ is an eigenvalue of $H_0$ and $\lambda \in J$. From the fact that $\lambda$ is a simple eigenvalue of $H_0$, we conclude
 that $\theta =c \varphi$ which contradicts the hypothesis $\langle \varphi ,  \psi \rangle  \neq 0$. Therefore, $\lambda$ cannot be an eigenvalue of $H_\kappa$.
\end{proof}

%***************************************************
\subsection{Existence of resonance} \label{res}

Now, we state the result concerning the existence of a resonance described by an approximate exponential decay of the survival probability. For the next theorem consider the situation where $H_0$  has a simple eigenvalue  $\lambda_0$, with $H_0 \varphi =\lambda _0 \varphi$. As above, given a
normalized vector  $\psi \in \mathcal H$  such that $\langle \varphi, \psi \rangle \neq 0$, we define
$H_\kappa := H_0 + \kappa | \psi \rangle \langle \psi | \, , \quad \kappa \in \mathbb{R}\, $.
 Let us  assume  there is an interval $I$ containing $\lambda _0$ such that $\mathcal F_{\kappa}(\lambda +i0)$ exists, for all $\lambda \in I$, where
\begin{equation}\label{functionF1}
\mathcal F_{\kappa} (z):= \langle \psi, P^{\perp } (H_{\kappa}^{\perp}-z)^{-1} P^{\perp }\psi \rangle
\end{equation}
and $P=\langle \varphi, \cdot\rangle \varphi$, $P^ \perp =I-P$. By $C^{1,\beta}(J)$ we denote the set of functions $F$ such that $F'$ is $\beta$-H\"older continuous in $J$, cf. section \ref{Boundary limit of Borel transforms}. Then we have,

\begin{theorem}\label{thmrankone}
Let $H_0 \varphi =\lambda _0 \varphi$ where $\lambda_0$ is a simple eigenvalue.
Assume that $\mathcal F_0(\lambda +i0) \in C^{1,\beta}(I)$, with $\beta > 0$  and it satisfies
\begin{equation}\label{fgr}
\mathcal \Im F_0(\lambda_0+i0) >0.
\end{equation}
Let $J \subset I$ be a closed interval such that $\lambda _0$ is an interior point of $J$ and $\mathcal \Im F_0(\lambda +i0) >0$, for all $\lambda \in J$.

Given $g$ a smooth characteristic function supported on $J$ which is identically $1$ in a neighborhood of $\lambda_0$, we have that for all $t \in \mathbb{R}$ and all small enough $\kappa \neq 0$, 
\begin{equation}\label{decexp}
\langle \varphi , e^{-iH_{\kappa}t}g(H_{\kappa}) \varphi \rangle = c_{\kappa} e^{-i\zeta_{\kappa}|t|} + R(t,\kappa)
\end{equation}
with $\zeta_{\kappa}$ and $c_{\kappa}$ described in Theorem \ref{protothm}.

In particular, $\Im \zeta_{\kappa} <0$, $c_{\kappa} = 1+O(\kappa^2 )$ and $| R(t,\kappa)| \leq C\kappa^{2}$. Moreover, $|t|\, |R(t,\kappa)|\le C\kappa^{2\alpha}$, if $\alpha \in (0,1)$ and
 $|t|\, |R(t,\kappa)|\le C\kappa^{2}|\log |\kappa ||$, if $\alpha =1$.
\end{theorem}

Before proving Theorem \ref{thmrankone}, we need the following lemmas.

Let $H$ be a self-adjoint operator defined in a Hilbert space $\mathcal{H}$ and $P$ be an orthogonal projector
with $\text{Ran}P \subset \text{Dom}(H)$, with Ran$P$ of finite dimension.
\begin{lemma}\label{fesh} (Feshbach-Livsic formula) Consider $ z \in \mathbb{C} , \Im\,  z \neq 0$. Then
the operator $ M:=P(H-z)P-PHP^{\perp}(H^{\perp} - z)^{-1}P^{\perp}HP $ is invertible and as an operator $M: \text{Ran}P \to  \text{Ran}P$ and
$$
P(H-z)^{-1}P = P\left(P(H-z)P-PHP^{\perp}(H^{\perp} - z)^{-1}P^{\perp}HP \right)^{-1}P \, .
$$
\end{lemma}
See \cite{How} for a proof.

\begin{lemma}\label{lemC1}
 Let $f \in C^{1,\beta}(I)$ where $I\subset \mathbb{R}$ an open interval. Assume that $f(x) \neq 0$ for all $x \in I$. Then $(1/f) \in C^{1,\beta }(J)$ where $J$
 is a closed interval contained in $I$.
\end{lemma}

\begin{proof}
Since $f(x) \neq 0$ for all $x \in I$, the function $\frac{1}{f}$ is $C^1$. To prove the $\beta$- H\"older continuity in $I$, we compute,
$$
(\frac{1}{f})'(x)- (\frac{1}{f})'(y)=\frac{f'(y)(f^2(x)-f^2(y))\,+ \,f^2(y)(f'(y)-f(x))}{f^2(x) f^2 (y)}
$$
For $x,y$ in $J$, by continuity the denominator is bounded away from zero.

On the other hand, by continuity of $f$ and $f'$ and the Mean Value Theorem, the first term in the numerator satisfies,
$$
|f'(y)|\, |f(x)+f(y)|\, |f(x)-f(y)|\le c |x-y|,
$$
Because Lipschitz continuity implies H\"older continuity, we conclude that the first term is $\beta$-H\"older continuous.

The second term is bounded above by
$$
c|f'(x)-f'(y)|
$$
which finishes the proof.
\end{proof}

\begin{proof}[Proof of Theorem \ref{thmrankone}] Following Krein's formula, see \cite{simon1} I.13 and \cite {don}, we obtain that
$$
{\mathcal{F}}_{\kappa}(z) = \frac{{\mathcal{F}}_0(z)}{1+\kappa \,{\mathcal{F}}_0(z)}\, .
$$
Therefore ${\mathcal{F}}_{\kappa}(\lambda+i0) $ exists and is finite for all $\lambda \in \bar{J}$ and $\kappa \in\newblock \mathbb{R}$.
By hypothesis and Lemma~\ref{lemC1} the function
$\lambda \to {\mathcal{F}}_{\kappa}(\lambda + i0)$ belongs to $C^{1,\beta}(J)$.
Actually, $(\lambda,\kappa) \to {\mathcal{F}}_{\kappa}(\lambda + i0)$ is a $ C^{1,\beta}(J\times \mathbb{R})$ function. Note that
for all  $ (\lambda,\kappa) \in   J\times \mathbb{R}$,
\begin{equation}\label{star}
\Im  {\mathcal{F}}_{\kappa}(\lambda + i0)  = \frac{ \Im {\mathcal{F}}_0(\lambda + i0)  }{|1+  \kappa  {\mathcal{F}}_0(\lambda + i0) |^2   } > 0
\end{equation}
Since $\lambda _0$ is simple, we can use the Feshbach-Livsic formula, see Lemma~\ref{fesh}, to obtain
\begin{eqnarray*}
\langle \varphi, (H_{\kappa} - z )^{-1} \varphi \rangle & = & \frac{1}
{\langle \varphi, (H_{\kappa}-z)\varphi \rangle - \langle  \varphi, PH_{\kappa}P^{\perp}(H_{\kappa}^{\perp}-z)^{-1} P^{\perp}H_{\kappa}P\varphi
 \rangle}\\
   & = & \frac{1}{\lambda_0 + \kappa \, |\langle \varphi, \psi \rangle  |^2  -z - \kappa^2 \,|\langle \varphi, \psi \rangle  |^2 {\mathcal{F}}_{\kappa}(z) }
\end{eqnarray*}
From the above identity and \eqref{fgr} we obtain that for all $\kappa \neq 0$, for all $\lambda \in \bar{J}$ the limit
$  \langle \varphi, (H_{\kappa} - \lambda -i0 )^{-1} \varphi \rangle $ exists  and
$\lambda \to \langle \varphi,(H_{\kappa} - \lambda -i0 )^{-1} \varphi \rangle $ is a $C^{1,\beta}(I)$ function.

Also, by the Stone's formula we know that
\begin{multline}\label{stone}
\langle \varphi, e^{-iH_{\kappa}t} g(H_{\kappa}) \varphi \rangle
  = \lim_{\epsilon \downarrow 0} \frac{1}{\pi}\, \int_{\mathbb{R}} g(\lambda ) e^{-i\lambda t} \Im  \langle  \varphi , (H_{\kappa}-\lambda - i\epsilon)^{-1}\varphi \rangle \, d\lambda \\
   =   \frac{1}{\pi}\, \int_{\mathbb{R}} g(\lambda ) e^{-i\lambda t} \Im  \left( \frac{1}{   \lambda_0 - \lambda + \kappa \,|\langle \varphi,\psi \rangle |^2   - \kappa^2 |\langle \varphi,\psi \rangle |^2   {\mathcal{F}}_{\kappa}(\lambda +i0) } \right)\, d\lambda \, .
\end{multline} since the limit can be taken inside the integral by \eqref{star}.

To finish the proof, apply Theorem~\ref{protothm} to \eqref{stone} with $\lambda_{\kappa}= \lambda_0 + \kappa |\langle \varphi,\psi \rangle |^2 $ and
$F(\lambda, \kappa) =  |\langle \varphi,\psi \rangle |^2   {\mathcal{F}}_{\kappa}(\lambda +i0)$.
\end{proof}

%***********************************************************
\subsection{Spectral Concentration and Sojourn time}\label{so}

The results of Section \ref{res} bear the two following straightforward consequences:
\begin{corollary}\label{spectralconc} Under the hypotheses of Theorem \ref{thmrankone}, we have that for any $t\in {\mathbb R},$
\begin{equation}\label{concentration}
\lim_{\kappa \to 0} \langle  \varphi ,  e^{-i \,\frac{1}{\kappa^2 \Gamma_{\kappa}} (H_{\kappa} -\Re \zeta_{\kappa})\, |t|}  g(H_{\kappa})\varphi \rangle = e^{- |t| }
\end{equation}
where $\Gamma_{\kappa} = |\langle \psi ,\varphi \rangle |^2 \,\Im \mathcal F(\lambda _{\kappa}^{\infty},\kappa)$.
\end{corollary}
\begin{proof}
By multiplying the equation (\ref{decexp}) by $e^{i\Re \zeta_{\kappa} |t|}$, we obtain,
$$
\langle \varphi , e^{-i(H_{\kappa}-\Re \zeta_{\kappa})|t|}g(H_{\kappa}) \varphi \rangle = c_{\kappa} e^{-\kappa^2 \Gamma_{\kappa} |t|} + e^{i\Re \zeta_{\kappa} |t|}R(t,\kappa)\, .
$$
After scaling the time, i.e. replacing $t$ by $\frac{t}{\kappa ^2 \Gamma_{\kappa}}$, it follows that
$$
\langle \varphi , e^{-i\,\frac{1}{\kappa^2 \Gamma_{\kappa}} (H_\kappa -\Re \zeta_{\kappa}) |t|}g(H_{\kappa}) \varphi \rangle = c_{\kappa} e^{- |t|} + e^{i \Re \zeta_{\kappa} \frac{|t|}{\Gamma_{\kappa}}}R(\frac{t}{\kappa^2\Gamma_{\kappa}}, \kappa)\, .
$$
The corollary follows now from the estimates on the error term in Theorem \ref{thmrankone}.
\end{proof}
\begin{remark} Note that $| \Gamma_{\kappa}-\Gamma_0| \leq C |\kappa |$ with $\Gamma_0 = \Im F(\lambda _0,0)$.
\end{remark}

Formula (\ref{concentration}) is related to Kato's spectral concentration, see \cite{davies}, \cite{green}, \cite{kato}. The connection between this type of formula and the spectral concentration has been established for isolated eigenvalues (of the unperturbed operator) in \cite{davies} and extended to embedded eigenvalues in \cite{green}.
%It actually expresses that after shifting the Hamiltonian $H_{\kappa}$ around $\Re \zeta_{\kappa}$ and scaling  the time by a factor $\kappa^2$, the system has an exponential asymptotic decay, as $\kappa$ tends to $0$.

Given any Hamiltonian $H$,  the quantity $|\langle \varphi, e^{-iHt}\varphi \rangle |^2$ measures the probability of finding the system in its initial state $\varphi$, at time $t$. Hence,
\begin{equation}\label{sojourn}
\tau_H (\varphi)  \equiv \int_{-\infty}^{\infty}|\langle \varphi, e^{-iHt}\varphi \rangle |^2 dt
\end{equation}
represents the expected amount of time the system spends in its initial state. As remarked in \cite{lavine2}, one expects that in presence of a resonance
near $\lambda $, there exists a state $\varphi$ whose sojourn time $\tau_H (\varphi)$ is very large and such that the spectral measure of $H$ in the state $\varphi$ is concentrated near $\lambda$. An explicit lower bound on the sojourn time appears in \cite{asch} also in the case of rank one perturations.

\begin{corollary}\label{sojourn} Under the hypotheses of Theorem \ref{thmrankone}, we have that for any $0< |\kappa |\leq \kappa_*,$
$$
\left| \tau_{H_{\kappa}} (\varphi_{\kappa})- \frac{1}{\kappa^2\Gamma_{\kappa}} \right| \le C\, \left\{
\begin{array}{lcc} |\kappa |^{2\alpha -1} & \text{if} & 0< \alpha < 1/2 \\
1 & \text{if} & \alpha \geq 1/2
\end{array}
\right. \, ,
$$
for some $C>0$, where $\varphi_{\kappa} := \sqrt{g}(H_{\kappa}) \varphi$ and $\Gamma_{\kappa} = |\langle \psi ,\varphi \rangle |^2 \,\Im \mathcal F(\lambda _{\kappa}^{\infty},\kappa)$.
\end{corollary}
Observe that $\varphi_{\kappa}\rightarrow \varphi$ when $\kappa\rightarrow0$.
\begin{proof} We denote by $\|\cdot \|_2$ the Hilbert norm on $L_2({\mathbb R},dt)$. For any $0< |\kappa |\leq \kappa_*,$
\begin{align*}
| \|{\mathcal I}(t,\kappa)\|_2^2\,-\, |c_{\kappa}|^2 \| e^{-i\zeta_{\kappa}|t|}\|_2^2 | &\le \|R(t,\kappa)\|_2 \left( \|{\mathcal I}(t,\kappa)\|_2\,+\, |c_{\kappa}| \|e^{-i\zeta_{\kappa}|t|}\|_2\right) \\
& \le \|R(t,\kappa)\|_2 \left( \|R(t,\kappa)\|_2\,+\, 2 |c_{\kappa} | \|e^{-i\zeta_{\kappa}|t|}\|_2\right) \, .
\end{align*}
where $\mathcal I$ is given by \eqref{quantity}.
Now, $\|e^{-i\zeta_{\kappa} |t|}\|_2^2=\frac{1}{\Gamma_{\kappa} \kappa^2}$. The result follows from the estimates on the error term in Theorem \ref{thmrankone}.
\end{proof}

%%%%%%%%%%%%%%%%%%

\section{Finite spectral multiplicity}

First, we recall the concept of spectral representation and  spectral multiplicity for a general self-adjoint operator.

%*********************************************************
\subsection{Reduced operator. Spectral measures}

Let $H: \mathcal{H} \to \mathcal{H}$ be a self-adjoint operator in a Hilbert space $\mathcal{H}$. From the spectral representation, see \cite{weidman2}  Theorem~7.18 p.195,
 we know that there exists a family of measures  $\{ \rho_{\alpha}: \alpha \in \Lambda \}$ and a unitary operator $\tilde{U}$ such that the following diagram commutes,
$$
\begin{CD}
\mathcal{H} @>H>> \mathcal{H} \\
@V\tilde{U}VV   @AA\tilde{U}^{-1}A \\
{\oplus}_{\alpha \in \Lambda} L^2({\mathbb{R}},\rho_{\alpha})  @>M_{id}>> {\oplus}_{\alpha \in \Lambda} L^2({\mathbb{R}},\rho_{\alpha})
\end{CD}
$$
That is, $H = \tilde{U}^{-1}M_{id}\tilde{U}$ where $M_{id}$ is the maximal operator of multiplication by the identity  $id$, i.e., $M_{id}f(x) = xf(x)$.

Assume that the operator $H$ has finite spectral multiplicity, that is $\Lambda$ is a finite set with $N$ elements. We
 can construct a matrix measure distribution $\rho$  and a space of vector valued functions $L_2({\mathbb{R}},{\mathbb{C}^N},\rho)$, such that
$H = U_H^{-1}M_{id} U_H$ where $U_H : \mathcal{H} \to L^2({\mathbb{R}},{\mathbb{C}^N},\rho)$ is unitary and $M_{id} f (x) = xf(x)$ is defined from
$L^2({\mathbb{R}},{\mathbb{C}^N},\rho) \to L^2({\mathbb{R}},{\mathbb{C}^N},\rho)$. For these concepts see \cite{ag} Vol.~2 section~71,  \cite{weidmann1}  section~10,  \cite{naimarkV1} p.101 , \cite{ag} Section 72 and \cite{weidmann1} Theorem~8.7.

%*******************************************
\subsection{Multiplicity of eigenvalues}

Since $H$ is unitarily equivalent to the operator multiplication for the identity
on $L_2({\mathbb{R}},\rho)$, with $\rho$ a matrix measure distribution, it is interesting to characterize the multiplicity of a given eigenvalue of $H$ in terms of $\rho$.
 For fixed $\lambda$, denote by  $\delta_\lambda$ the measure
\begin{equation}
  \label{eq:dirac-measure}
  \delta_\lambda(\Delta):=
  \begin{cases}
    1 & \lambda\in\Delta\\
    0 & \lambda\not\in\Delta
  \end{cases}
  \end {equation}
where  $\Delta\subset\mathbb{R} $ is a Borel set. In what follows $M = (m_{ij})$ shall denote a non-negative symmetric  matrix with $m_{ij}\in\mathbb{C}$ where $i,j= 1,2,...N$.  We shall use the notation $M\delta_{\lambda}$ for the  matrix measure distribution with entries $ (m_{ij}\delta_{\lambda})$.
\begin{lemma}
Consider  $\lambda_0 \in \mathbb{R}$, $\mu$ a matrix measure distribution
 defined on $\mathbb{R}$, with $ \mu( \{ \lambda_0 \})=0$ and $M $ a non-negative constant matrix as above.
  Let us take  $M_{id} $ defined on the space $L^2(\mathbb{R},\mathbb{C}^N, d\mu+M\delta_{\lambda_0})$.

  Then $\lambda_0$ is an eigenvalue of $M_{id}$ of multiplicity $p \in \{1,2,\ldots, N \}$ if and only if
  $\text{Rank} M = p$ if and only if $ \text{dim (Ker$M$)} =N-p $.
  \end{lemma}

\begin{proof}
$\lambda_0$ is an eigenvalue of $M_{id}$ if and only if there exists a non zero vector $\varphi \in  L^2(\mathbb{R},\mathbb{C}^N, d\mu+M\delta_{\lambda_0})$
such that $(M_{id}-\lambda_0)\varphi (x)=0$ almost everywhere. This implies the existence of  $\vec{u}\in \mathbb{C}^N$
such that the eigenfunction has the form  $ \varphi(x)=\chi_{\{ \lambda_0 \} }(x) \vec{u}$, where $\chi_{\{ \lambda_0 \} }(x)=1$ if $x=\lambda_0$ and $0$ otherwise.
Thus $\lambda_0$ is an eigenvalue of $M_{id}$ if and only if $\|\varphi\|_{L_2}^2=\langle  \vec{u}, M\vec{u} \rangle \neq 0 $ and $\varphi = \chi_{\{ \lambda_0 \} }(x) \vec{u}$. Now,
for the difference of two eigenvectors we have
$\|  \chi_{\{ \lambda_0 \} }(x) \vec{u}-\chi_{\{ \lambda_0 \} }(x) \vec{v} \|_{L_2}^2 =\langle  \vec{w}, M\vec{w} \rangle$
with $\vec{w}=\vec{u}-\vec{v}$. Since $M$ is non-negative,$\langle  \vec{w}, M\vec{w} \rangle =0$ if and only if $Mw=0$. Therefore the eigenvector $\varphi = \chi_{\{ \lambda_0 \} }(x) \vec{u}$
is different from the eigenvector $\chi_{\{ \lambda_0 \} }(x) \vec{v}$ in the space $L^2(\mathbb{R},\mathbb{C}^N, d\mu+M\delta_{\lambda_0})$ if and only if $M\vec {u}\neq M \vec {v}$. Now
let us construct an isomorphism between the subspace of eigenvectors and the Range of $M$ in the following way:
$$I:Ker (M_{id}-\lambda_0)\rightarrow \{ \vec{v}: M\vec{u}=\vec{v}\}$$$$I( \chi_{\{ \lambda_0 \} }(x) \vec{u})=M\vec{u}$$
 Since $I$ is biyective and linear the two spaces have the same dimension and the lemma is proven.

\end{proof}
\begin{remark}
If rank $M>1$, then the matrix distribution $d\mu+M\delta_{\lambda_0} $ does not correspond to a Sturm-Liouville operator with limit point case conditions, since the singular spectrum of such operators is simple see \cite {simon3}.
\end{remark}

%***********************************
\subsection{Reduction process}

Let $B \subset {\mathbb{R}}$ be a fixed  Borel set and let
$P_{B}:= E_{H}(B)$ where $E_{H}$ is the spectral family associated to $H$.
Let us decompose the matrix function $\rho$ as follows:
$$
\rho(\Delta) = \rho_{B} (\Delta) + \rho_{B^c} (\Delta)
$$
with $\Delta \subset {\mathbb{R}}$ a Borel set where
$$
\rho_{B} (\Delta):= \rho(\Delta\cap B)\, , \quad
\rho_{B^c} (\Delta):= \rho(\Delta\cap B^c)\,.
$$
The following  result will be useful to prove almost exponential decay of resonant states.

\begin{theorem}\label{reduce1}
Let   $\psi$ be a vector in $ \mathcal{H}$ and $H$ be a self-adjoint operator with finite spectral multiplicity  $N$. Then
\begin{multline}
\langle P_{B}\psi, (P_{B}HP_{B}-z)^{-1} P_{B}\psi \rangle  =  \int_{\mathbb{R}} \frac{\langle (U_H \psi)(x), d\rho_B(x) (U_H \psi)(x) \rangle _{\mathbb{C}^N}}{x - z} \nonumber \\
   =  \int_{\mathbb{R}} \frac{\langle (U_H P_{B}\psi)(x) , d\rho(x)(U_H P_{B}\psi)(x) \rangle _{\mathbb{C}^N} \,}{x- z} .
\end{multline}
\end{theorem}
\begin{remark}
In case de matrix $d\rho(x)$ has only one entry
$$\langle (U_H \psi)(x) , d\rho(x)(U_H \psi)(x) \rangle  =|(U_H \psi)(x)|^2 d\rho(x).$$

\end{remark}
Before we prove the above theorem we need the following lemmas.

\begin{lemma}\label{lem10}
For any Borel set $\Delta \subset {\mathbb{R}} $ the following identity holds,
\begin{equation}
P_{B} E_{ P_{B}H P_{B}}(\Delta) P_{B} = E_{H}(\Delta \cap B) = P_{B}E_{H}(\Delta) P_{B}\, .
\end{equation}
\end{lemma}
\begin{proof}
We first note that
$
 P_{B} H  P_{B}= H P_{B}=f(H)$ where $f(x) = x \chi_{B}(x)$ , with  $\chi_{B}$ the characteristic  function on the Borel set $B$.
 Also we know that  $ E_{f(H)} (\Delta) = E_H (f^{-1}(\Delta) )$, see \cite{birman1} Theorem~4, Chapter~6 p.158. Thus,
  \begin{eqnarray*}
 P_{B} E_{P_{B} H  P_{B}}(\Delta) P_{B} & = &  E_{P_{B}  HP_{B}}(\Delta)P_{B} =  E_{HP_{B}} (\Delta)  P_{B} \\
 & = &  E_{H P_{B}}(\Delta) E_H(B)
   = E_H(f^{-1}(\Delta))  E_H(B)\\
   & = & E_H(f^{-1}(\Delta)\cap B)\, .
\end{eqnarray*}
We claim that for $f(x) = x \chi_{B}(x)$ one has that  $f^{-1}(\Delta) \cap B = \Delta \cap B$.
Clearly, if $x \in \Delta \cap B$ then $f(x)=x\chi_{B}(x) = x \in \Delta$, so $x\in f^{-1}(\Delta)$.

Also, $x\in f^{-1}(\Delta) \cap  B$,   implies that $f(x) \in \Delta $. But for $x \in B$,  $g(x) = 1$ so
$f(x)=x $ and $x \in \Delta$, ending the proof.
\end{proof}

\begin{lemma}\label{lem11}
For any vector $\psi \in \mathcal{H}$
\begin{equation}
\langle  E_{H}(\Delta) \psi, \psi    \rangle
=\int_{\Delta}    \langle (U_H \psi)(x), d\rho(x) (U_H \psi)(x) \rangle \ .
\end{equation}
\end{lemma}

\begin{proof}

From the spectral representation theorem we know that
$ \chi_{\Delta} (H) = U_H^{-1} M_{\chi_{\Delta}}U_H $
where
$(M_{\chi_{\Delta}}h)(x) = \chi_{\Delta}(x) h(x)$. Then
\begin{eqnarray*}
\langle E_H(\Delta)\psi, \psi    \rangle_{\mathcal{H}}
 & = &
\langle \chi_{\Delta}(H) \psi, \psi    \rangle_{\mathcal{H}}
=
\langle  U_H \, \chi_{\Delta}(H) \psi, U_H \psi  \rangle_{L_2(\rho)} \\
  & = &
  \langle M_{\chi_{\Delta}} U_H \psi, U_H \psi  \rangle_{L_2(\rho)}\\
  & = &\int_{\mathbb{R}} \chi_{\Delta}(x)  \langle (U_H \psi)(x), d\rho(x) (U_H \psi)(x) \rangle  \\
  & = & \int_{\Delta}  \langle (U_H \psi)(x), d\rho(x) (U_H \psi)(x) \rangle \ .
\end{eqnarray*}
\end{proof}

\begin{proof}[Proof of Theorem \ref{reduce1}] Let us define the measure
$\mu_g(\Delta) = \langle  E_{H}(\Delta \cap B)g,g \rangle $. According to Lemma~\ref{lem11} and the definition of $\rho_{B}(\Delta)$ we deduce the following identity:
\begin{multline*}
\mu_\psi (\Delta) = \int_{\Delta \cap B} \langle (U_H \psi)(x), d\rho(x) (U_H \psi)(x) \rangle
= \int_{\Delta }  \langle (U_H \psi)(x),  d\rho_{B} (U_H \psi)(x) \rangle  \,(x)  .
\end{multline*}
On the other hand, using Lemma~\ref{lem10} we get that
\begin{multline}\label{mu11}
\mu_\psi (\Delta) = \langle  P_{B} E_{ P_{B}H P_{B}}(\Delta)  P_{B}\psi,\psi   \rangle
 =
\int_{\Delta}  \langle (U_H \psi)(x), d \rho_{B}(x) (U_H \psi)(x) \rangle\
\end{multline}
By the spectral theorem,
\begin{eqnarray*}
\langle P_{B} \psi, ( P_{B}HP_{B}-z)^{-1}
P_{B}\psi  \rangle
& = & \int_{\mathbb{R}}  \frac{1}{x-z} \, d\langle   E_{ P_{B} H P_{B}} P_{B} \psi , P_{B} \psi  \rangle \\
 & = & \int_{\mathbb{R}}  \frac{1}{x-z} \, d\langle  P_{B}  E_{ P_{B} H P_{B}} P_{B} \psi , \psi  \rangle \, .
\end{eqnarray*}
So, by identity (\ref{mu11}) one has that
\begin{multline*}
\langle P_{B} \psi, ( P_{B}HP_{B}-z)^{-1}
P_{B}\psi  \rangle
 =
\int_{\mathbb{R}} \frac{1}{x-z}\, d\mu_{\psi}(x) \\
 =
\int_{\mathbb{R}} \frac{1}{x-z}\,  \langle (U_H \psi)(x), d \rho_{B}(x) (U_H \psi)(x) \rangle \, .
\end{multline*}
To prove the second equality in the theorem we apply again Lemma~\ref{lem10} to obtain that
\begin{eqnarray*}
\langle P_{B} \psi , ( P_{B} H P_{B}-z)^{-1}
P_{B}\psi  \rangle
 & = &
\int_{\mathbb{R}} \frac{1}{x-z}\, d\langle P_{B}E_{P_{B}HP_{B}  }P_{B}\psi, \psi \rangle  \\
   & = & \int_{\mathbb{R}} \frac{1}{x-z}\, d \langle E_H(x)P_{B}\psi, P_{B}\psi \rangle \, .
\end{eqnarray*}
We finish the proof by using Lemma~\ref{lem11}.
\end{proof}

Let $H_0$ be a self-adjoint operator defined on $\mathcal{H}$ with spectrum of finite multiplicity $N$.
Define
$$
H_{\kappa} = H_0 + \kappa |\psi\rangle \langle \psi|
$$
where $\psi$ a normalized vector on $\mathcal{H}$ and $\kappa \in \mathbb{R}$.
Let $U_{H_0}$ be  a unitary operator and $\mu$ a matrix measure distribution  defined on $\mathbb{R}$ such that
$U_{H_0}H_0 U_{H_0} ^{-1}= M_{id}$ on $L^2(\mathbb{R},\mathbb{C}^N, d\mu)$.

Now we can present a useful result about almost exponential decay. It will be applied in next section to Sturm-Liouville operators.

\begin{theorem}\label{abstract}
\begin{enumerate}
\item[a)]Assume that $H_0$ has a simple eigenvalue $\lambda_0$ embedded in some continuous spectrum. Precisely,
 \begin{itemize}
  \item[a1)] $  H_0 \varphi = \lambda_0 \varphi $  , $\| \varphi \|=1$.
  \item[a2)] There exists an open interval $I$ of $\mathbb{R}$ with $d\mu = \gamma(\lambda) d\lambda + M\delta_{\lambda_0}$ , $\lambda_0 \in I$ where for each $\lambda \in I$,
   $\gamma(\lambda)$ is a non-negative  matrix and $M$ is a constant non-negative matrix of rank one.   \end{itemize}
\item[b)] The map $\lambda \to \langle \,  (U_{H_0}\psi)(\lambda) , \gamma(\lambda ) (U_{H_0}\psi)(\lambda)  \, \rangle $ is $C^{1,\alpha}(I)$, with
 $\, 0 < \alpha \leq 1$.
\item[c)] For $\varphi$ and $\psi$ we assume that
 \begin{multline*}
\qquad \langle  \varphi, \psi\rangle_{{\mathcal{H}}}  =  \langle U_{H_0} \varphi, U_{H_0}\psi\rangle_{L^2 } \neq 0
%\Longleftrightarrow \langle M\vec{u}, (U_{H_0}\psi )(\lambda_0)\rangle_{\mathbb{C}^N} \neq 0 \\
   \text{ and }  \langle (U_{H_0} \psi)(\lambda_0), \gamma(\lambda_0)  (U_{H_0}\psi) (\lambda_0 )\rangle_{\mathbb{C}^N} \neq 0.
 \end{multline*}
\end{enumerate}

\noindent Then there exists an open interval $J$, $\lambda_0 \in J, \bar{J}\subset I$ such that
\begin{enumerate}
\item For all $\kappa \neq 0$ , $\sigma_{p}(H_{\kappa})\cap J = \emptyset$, $ J \subset\sigma_{ac}$.
\item Given $g$ a smooth characteristic function supported on $J$ which is identically $1$ in a neighborhood of $\lambda_0$, we have that for all $t \in \mathbb{R}$ and all $\kappa \neq 0$ but small enough
\begin{equation*}
\langle \varphi , e^{-iH_{\kappa}t}g(H_{\kappa}) \varphi \rangle = c_{\kappa} e^{-i\zeta_{\kappa}|t|} + R(t,\kappa)
\end{equation*}
with $\zeta_{\kappa}$ and $c_{\kappa}$ described in Theorem \ref{protothm}.
\end{enumerate}
In particular, $\Im \zeta_{\kappa} <0$, $c_{\kappa} = 1+O(\kappa^2 )$ and $| R(t,\kappa)| \leq C\kappa^{2}$. Moreover, $|t| |R(t,\kappa)|\le C\kappa^{2\alpha}$, if $\alpha \in (0,1)$ and
 $|t| |R(t,\kappa)|\le C\kappa^{2}|\log |\kappa ||$, if $\alpha =1$.
\end{theorem}

\begin{remark}\label{vector}
There is  a normalized vector $\vec{u}$ such that
   $\text{Ran}(M)=\mathbb{C}\vec{u}$ and
   $$
   U_{H_0}\varphi = \frac{1}{\sqrt{\langle \vec{u},  M \vec{u} \rangle_{{\mathbb{C}}^N}  }} \, \chi_{\{ \lambda_0 \}}\vec{u}\, .
   $$

\end{remark}
Now we can use the tools we have developed to prove the theorem.
\begin{proof}
Applying Theorem~\ref{reduce1} with $B=\mathbb{R}\setminus\{\lambda_0\}$ and $\rho = \mu$ we obtain
$$
\langle \psi, P^{\perp} (H_0^{\perp}-z)^{-1 } P^{\perp} \psi  \rangle
= \int \frac{\langle (U_{H_0} \psi)(\lambda) ,\gamma(\lambda) (U_{H_0} \psi)(\lambda)\rangle_{\mathbb{C}^N}d\lambda }{\lambda - z}\,  .
$$
Then (1) follows applying Theorem~\ref{pp}, taking into account that
$$\mathcal \Im F_0(\lambda_0+i0)=\Im\langle \psi, P^{\perp} (H_0^{\perp}-\lambda_0-i0)^{-1 } P^{\perp} \psi  \rangle =\\
 \langle (U_{H_0} \psi)(\lambda_0) ,\gamma(\lambda_0) (U_{H_0} \psi)(\lambda_0)\rangle_{\mathbb{C}^N}$$ and (2) follows from Theorem~\ref{thmrankone} together with Theorem~\ref{maintheorem2}.
\end{proof}

%%%%%%%%%%%%%%%%%%%%%%%%%%%%%%%%%

\section{A Sturm Liouville model with an embedded eigenvalue}

First let us construct an operator  acting on functions defined on the half axis $\mathbb R^+$ which has an embedded eigenvalue.
The  function $$ h(x):= \cos x+k\sin x$$
 is solution of the initial value problem
 \begin{eqnarray*}
-u''&=&u\\ u(0)&=&1, \quad u'(0)=k \,.
\end{eqnarray*}
  We assume $k>0$ fixed, (one can set for example $k=1$).
Let us define
\begin{eqnarray*} \label{q}
 q(x):= -2k \frac{d}{dx}\Big[\frac{h^2(x)}{1+k\int_0^x h ^2(t)dt}\Big]\, .
 \end{eqnarray*}
The operator $L_N$  ($N$ stands for the Neumann boundary condition) acting on a dense subspace of $L_2(\mathbb R^+)$ generated by
$$
(lu)(x)=-u''+ q(x)u, \quad u'(0)=0 , \quad 0\leq x<\infty
$$
has the following spectral function see \cite{levitan} p. 46:
\begin{equation}
\rho _N(\lambda )=\hat{\rho}(\lambda )+ks(\lambda -1)
\end{equation}
 where
$s(t)$ \small{$=\left\{\begin{array}{r@{\quad \mbox {if} \quad}l} 0 & t< 0\\1 & t\geq 0 \end{array} \right. $}
and $\hat{\rho}(\lambda )$ is the spectral function of the operator $L_k$ generated by
\begin{eqnarray*}
(l_0 u)(x)&=&-u'', \quad 0\leq x<\infty \\
u'(0)&=&ku(0)
\end{eqnarray*}
that is
  \begin{equation}
d \hat{\rho}(\lambda )=\left\{\begin{array}{r@{\quad \mbox {if} \quad}l} \frac{\sqrt{\lambda }}{\pi (\lambda +k^2)}d\lambda  & \lambda \geq  0\\0\qquad & \lambda <0 \end{array} \right.
\end{equation}
See \cite{naimarkV1} p. 141. Therefore the point  $1$ is an embedded eigenvalue for the operator $L_N$.

The operator $L_N$ defined as above will take the place of $H_0$ in Theorem~\ref{abstract}. This operator satisfies the hypothesis a1) and a2) of that theorem. Fix a vector $\psi \in L^2(\mathbb{R}^+)$, $\mathbb{R}^+ = \{ x \in \mathbb{R} : x\geq 0 \}$. Consider
the perturbed operator
$$
H_{\kappa} = L_N + \kappa |\psi\rangle \langle \psi|
$$
Let $\mathcal{O} \subset \mathbb{R}^+$ an open interval such that $\lambda_0 = 1 \in \mathcal{O}$. To verify hypothesis b) of Theorem~\ref{abstract} we shall find an interval $I \subset \mathbb{R}^+$ such that
the eigenvalue $\lambda_0 = 1 \in I$ and $|U_{L_N}\psi(\lambda)|^2 \hat{\rho}(\lambda) $ is in $C^{1,\alpha}(I)$ with $ 0< \alpha \leq 1$.

The unitary operator $U_{L_N}$ is given by the transform
$$
(U_{L_N}\psi)(\lambda) = \int_{{\mathbb{R}}^+}  \omega (t,\lambda) \psi(t) \, dt
$$
where $\omega(t,\lambda)$ is a solution of the eigenvalue problem
$$
-u'' + q(x) u = \lambda u  \, \quad u(0,\lambda) = 0 \, , \, u'(0,\lambda)=1
$$
for any $\lambda \in \mathbb{C}$.  The function $\omega(t,\lambda)$ is an entire function of $\lambda$.

If we choose $\psi$ to be of compact support then
$(U_{L_N}\psi)(\lambda) \in C^{\infty}(\mathcal{O}) $. Since $\hat{\rho} \in C^{\infty}(\mathcal{O})$ then
$|(U_{L_N}\psi)(\lambda)|^2 \hat{\rho}(\lambda)  \in C^{\infty}(\mathcal{O})$ and therefore there is an open interval $I$
containing $\lambda_0 = 1$ such that $|(U_{L_N}\psi)(\lambda)|^2 \hat{\rho}(\lambda) \in C^{1,\alpha}(I) $, so condition b)
 in Theorem~\ref{abstract} is satisfied.

To realize condition c) of Theorem~\ref{abstract} we can take for instance $\psi(t) = \chi_{\Delta}(t) \omega(t,1)$ where $\Delta $ is an open interval. So,
$$
|(U_{L_N}\psi)(1)|^2 = \int_{\Delta} |\omega(t,1)|^2 dt > 0\,
$$
and therefore$$ |(U_{L_N}\psi)(1)|^2 \hat{\rho}(1) >0  $$
 Since $\hat{\rho}(1) =\frac{1}{\pi(1+k^2)}>0$.
 Moreover
 \begin{eqnarray*}
 \langle  \varphi, \psi\rangle_{{\mathcal{H}}}  =  \langle U_{L_N} \varphi, U_{L_N}\psi\rangle_{L^2 }
 &=&
 \int_{\mathbb{R}}\chi_{\{1\}} (x)(U_{L_N}\psi)(x)d\rho _N(x)\\
 &=&(U_{L_N}\psi)(1)\rho _N(\{1\})
 =(U_{L_N}\psi)(1)k\ne0.
 \end{eqnarray*}

%%%%%%%%%%%%%%%%%%%%%%%%%%%%%%%%%%%%%%%%%%

\section{Boundary limit of Borel transforms}\label{Boundary limit of Borel transforms}

Let $I \subset \mathbb{R}$ be an open interval and consider a measure
$\mu_{I} : {\mathcal{B}} \to {\mathbb{R}}^+$ defined on the Borel sets $\mathcal{B}$ of $\mathbb{R}$
such that $\mu_{I} $ restricted to $I$ is absolutely continuous with respect to Lebesgue measure, i.e. there exists a measurable function
$f: {\mathbb{R}} \to {\mathbb{R}}^+$ such that for any Borel set $\Delta \subset I  $, it holds that
$$
\int_{\Delta} f(x) dx = \mu(\Delta)\, .
$$

Assume moreover that $f \in C^{1,\alpha}(I)$, that is, $f \in C^1(I)$  with $f'$ $\alpha-$H\"older continuous on $I$, $0 < \alpha \leq 1$. We write $\mathfrak{F}(z)$ for the corresponding {\it  Borel transform} associated  to the measure $\mu_I$,
$$
\mathfrak{F}(z) = \int \frac{d\mu_{I}(x)}{x-z} \, , \quad \Im\,  z \neq 0 \, .
$$
The following results assures the existence and the smoothness of the  boundary values $\lim_{\epsilon \to 0^+} \mathfrak{F}(\lambda + i \epsilon)$,
See \cite{Yaf}, \cite{Bel} \cite{MR2778944},\cite{courant}.
\begin{theorem}\label{maintheorem2}
Let  $I \subset \mathbb{R}$ be an open interval and $\mu_{I}$ the measure
defined above. Then
$$
F(\lambda) := \lim_{\epsilon \to 0^+} \mathfrak{F}(\lambda + i \epsilon)
$$
 exists and $F\in C^1(I)$. Moreover for any interval $J $ such that $\bar{J}\subset I$, the function $F'$ is $\beta$-H\"older continuous in $J$ for all $\beta<\alpha$.,
 that is, $F \in C^{1,\beta}(J)$ .
\end{theorem}

\begin{proof}
Let $ \mathcal{K}$ be an open subinterval such that
$\bar{J} \varsubsetneq \mathcal{K} \varsubsetneq \bar{\mathcal{K}} \varsubsetneq I$
and let us define  $\eta \in C_0^{\infty}(\mathbb{R}, [0,1]) $
with $\eta(x) =0$ if $x \notin \mathcal{K}$   and
$\eta(x) =1$ if $x \in \bar{J}$. For any  $z \in {\mathbb{C}}$ with
$\Im\,  z \neq 0 $,
$$
\mathfrak{F}(z) = \int_{ \mathcal{K} } \frac{f(x) \eta(x)}{x-z} \, dx
 +   \int_{ \mathbb{R}- \bar{J} }\frac{(1-\eta(x))}{x-z} \, d\mu(x)\, .
$$
If $\text{Re}\,  z \in J$ the second integral in the above
equality represents an holomorphic function, therefore this term is $C^1(I)$ and
$\alpha$-H\"older continuous in any subinterval $J$ with $\bar{J}\subset I$.
(Recall that continuously differentiable function on compact sets of the real line are $\gamma$-H\"older continuous with $0<\gamma \leq 1$).

 Now let us study the first term. Setting  $z=\lambda + i \epsilon $, it follows  that
$$
\int_{ \mathcal{K} } \frac{f(x) \eta(x)}{(x-\lambda) - i \epsilon} \, dx
=
\mathfrak{R}(\lambda,\epsilon) + i  \mathfrak{I}(\lambda,\epsilon)
$$
where
$$
{\mathfrak{R}}(\lambda,\epsilon) := \int_{ \mathbb{R} }
\frac{f(x) \eta(x)(x-\lambda) }{(x-\lambda)^2 + \epsilon^2} \, dx
\, , \quad
{\mathfrak{I}}(\lambda,\epsilon) := \int_{ \mathbb{R} }
\frac{\epsilon f(x) \eta(x) }{(x-\lambda)^2 + \epsilon^2} \, dx\, .
$$
Since $f\eta \in C_0^1(\mathbb{R})$ we obtain  that
$
I(\lambda):=\lim_{\epsilon \to 0^+} \mathfrak{I}(\lambda,\epsilon)=
\pi f(\lambda) \eta(\lambda)
$, see \cite{pearson} Lemma~2.3 i) p.41, so $I(\lambda) \in C_0^1(\mathbb{R})$.

Moreover, the derivative $I'(\lambda)= \pi (f(\lambda) \eta(\lambda) )'$ is $\alpha$-H\"older and thus $\beta$-H\"older in$J$ with $\beta<\alpha$. This follows because
 $(f\eta )' = f' \eta + f \eta'$ and $f\eta' \in C_0^1(I)$ and therefore $\alpha$-H\"older in $J$ and
\begin{eqnarray*}
|f'(x)\eta(x) - f'(y)\eta (y)| & \leq & |\eta(x) (f'(x)- f'(y))|+
|(\eta(x)- \eta(y))f'(y)|\\
       & \leq & M|x-y|^{\alpha}
\end{eqnarray*}
because $f'$ is $\alpha$-H\"older and $\eta \in C_0^{\infty}$.

Now let us consider $R(\lambda) := \lim_{\epsilon \to 0^+} {\mathfrak{R}}(\lambda,\epsilon)$.
This limit exists because $f\eta \in C_0^1$ and therefore $f\eta $ is $\alpha$-H\"older continuous with compact support, see Lemma 10 in\cite{Bel}. Moreover,
$$
R(\lambda) = P.V. \int_{\mathbb{R}} \frac{f(x) \eta(x) }{x-\lambda}:=
{\mathfrak{H}}[f\eta] (\lambda)
$$
where P.V. means the integral in the sense of principal value and $\mathfrak{H}$ stands for the Hilbert transform, see  \cite{pearson} Lemma~2.5, p.51and \cite{courant}.
Since $f\eta  \in C^1$ and  $f\eta \in L_p$ one deduces that
$({\mathfrak{H}}[f\eta](\lambda))' = \mathfrak{H}[(f\eta)'](\lambda)$, see \cite{pandey} Theorem~1 and formula~3.24.
Using again Privaloff-Korn's Theorem or Lemma 10 in\cite{Bel} and the $\alpha$-H\"older continuity of $(f\eta)'$  we finally obtain that
 $\mathfrak{H}[(f\eta)'](\lambda)$ is $\beta$-H\"older continuous and thus $(R(\lambda))'$ is $\beta$-H\"older continuous.
\end{proof}

%%%%%%%%%%%%%%%%%%%%%%%%%%%%%%%%%%%%%%%%

\subsection*{Acknowledgments}
R. del Rio would like to thank the warm hospitality of the Pontificia Universidad Cat\'olica de Chile. Also, C. Fern\'andez expresses his gratitude to UNAM, M\'exico, for the
warm hospitality. R. del Rio thanks M. Ballesteros for useful comments.

%\bibliography{experbiblio.bib}

%\bibliographystyle{amsplain}

\end{document}